\documentclass[oneside]{amsart}
\usepackage[utf8]{inputenc}
\usepackage{color}
\usepackage{mathrsfs}
\usepackage{bm}
\usepackage{amsthm}
\usepackage{amstext}
\usepackage{amssymb}
\usepackage[unicode=true,pdfusetitle,
 bookmarks=true,bookmarksnumbered=true,bookmarksopen=true,bookmarksopenlevel=2,
 breaklinks=false,pdfborder={0 0 1},backref=false,colorlinks=true]
 {hyperref}

\makeatletter
\numberwithin{equation}{section}
\numberwithin{figure}{section}
\theoremstyle{plain}
\newtheorem{thm}{\protect\theoremname}[section]
  \theoremstyle{remark}
  \newtheorem{rem}[thm]{\protect\remarkname}
  \theoremstyle{plain}
  \newtheorem{cor}[thm]{\protect\corollaryname}
  \theoremstyle{remark}
  \newtheorem*{acknowledgement*}{\protect\acknowledgementname}
  \theoremstyle{definition}
  \newtheorem{defn}[thm]{\protect\definitionname}
  \theoremstyle{plain}
  \newtheorem{lem}[thm]{\protect\lemmaname}
  \theoremstyle{plain}
  \newtheorem{prop}[thm]{\protect\propositionname}
  \theoremstyle{plain}
  \newtheorem{conjecture}[thm]{\protect\conjecturename}

\usepackage{a4wide}

\usepackage{version}
\newenvironment{NB}{\color{red}{\bf N.B.} \footnotesize}{}

\excludeversion{NB}
\excludeversion{NB2}

\newcommand{\Uq}{\mathbf{U}_q}
\newcommand{\wt}{\operatorname{wt}}

\newcommand{\up}{\mathrm{up}}
\newcommand{\low}{\mathrm{low}}

\newcommand{\GLS}{Gei\ss-Leclerc-Schr\"{o}er~}
\newcommand{\PBW}{Poincar\'{e}-Birkhoff-Witt~}

\makeatother

  \providecommand{\acknowledgementname}{Acknowledgement}
  \providecommand{\conjecturename}{Conjecture}
  \providecommand{\corollaryname}{Corollary}
  \providecommand{\definitionname}{Definition}
  \providecommand{\lemmaname}{Lemma}
  \providecommand{\propositionname}{Proposition}
  \providecommand{\remarkname}{Remark}
\providecommand{\theoremname}{Theorem}

\begin{document}

\title{Remarks on Quantum unipotent subgroup and Dual canonical Basis}
\begin{abstract}
We prove the tensor product decomposition of the half of quantized
universal enveloping algebra associated with a Weyl group element
which was conjectured by Berenstein and Greenstein \cite[Conjecture 5.5]{BG:double}
using the theory of the dual canonical basis.
\end{abstract}

\author{Yoshiyuki Kimura}



\thanks{This work was supported by the JSPS Program for Advancing Strategic
International Networks to Accelerate the Circulation of Talented Researchers
“Mathematical Science of Symmetry, Topology and Moduli, Evolution
of International Research Network based on OCAMI” and JSPS Grant-in-Aid
for Scientific Research (S) 24224001.}

\curraddr{Department of Mathematics, Graduate School of Science, Kobe University,
1-1, Rokkodai, Nada-ku, Kobe 657-8501, Japan}

\email{ykimura@math.kobe-u.ac.jp}

\maketitle

\section{Introduction}

Let $\mathfrak{g}$ be a symmetrizable Kac-Moody Lie algebra and $w$
be a Weyl group element. In \cite{Kim:qunip}, we have studied the
compatibility of the dual canonical basis and the quantum coordinate
ring of the unipotent subgroup associated with a finite subset $\Delta_{+}\cap w\Delta_{-}$.
The purpose of this paper is to study the compatibility of the dual
canonical basis and the ``quantum coordinate ring'' of the pro-unipotent
subgroup associated with a co-finite subset $\Delta_{+}\cap w\Delta_{+}$
and show the multiplicity-free property of the multiplications of
the dual canonical basis between finite part and co-finite part.

The following tensor product decomposition of the half $\mathbf{U}_{q}^{-}$
was conjectured by Berenstein and Greenstein \cite[Conjecture 5.5]{BG:double}
in general. We also prove the decomposition in the dual integral form
$\Uq^{-}\left(\mathfrak{g}\right)_{\mathcal{A}}^{\up}$ of Lusztig
integral form $\mathbf{U}_{q}^{-}\left(\mathfrak{g}\right)_{\mathcal{A}}$
with respect to Kashiwara's non-degenerate bilinear form.
\begin{thm}
\label{thm:dec}

Let $T_{w}=T_{i_{1}}T_{i_{2}}\cdots T_{i_{\ell}}\colon\Uq\to\Uq$
be the Lusztig braid group action associated with a Weyl group element
$w$, where $\bm{i}=\left(i_{1},\cdots,i_{\ell}\right)$ is a reduced
word of $w$. 

\textup{(1)} For a Weyl group element $w\in W$, multiplication in
$\Uq^{-}$ defines an isomorphism of vector spaces over $\mathbb{Q}\left(q\right)$:
\begin{align*}
\left(\Uq^{-}\cap T_{w}\Uq^{\geq0}\right)\otimes\left(\Uq^{-}\cap T_{w}\Uq^{-}\right) & \xrightarrow{\sim}\Uq^{-}.
\end{align*}

\textup{(2)} For a Weyl group element $w\in W$, we set $\left(\Uq^{-}\cap T_{w}\Uq^{\geq0}\right)_{\mathcal{A}}^{\up}:=\Uq^{-}\left(\mathfrak{g}\right)_{\mathcal{A}}^{\up}\cap T_{w}\Uq^{\geq0}$
and $\left(\Uq^{-}\cap T_{w}\Uq^{-}\right)_{\mathcal{A}}^{\up}:=\Uq^{-}\left(\mathfrak{g}\right)_{\mathcal{A}}^{\up}\cap T_{w}\Uq^{-}$.
Then the multiplication in $\Uq^{-}\left(\mathfrak{g}\right)_{\mathcal{A}}^{\up}$
defines an isomorphism of free $\mathcal{A}$-modules:
\[
\left(\Uq^{-}\cap T_{w}\Uq^{\geq0}\right)_{\mathcal{A}}^{\up}\otimes_{\mathcal{A}}\left(\Uq^{-}\cap T_{w}\Uq^{-}\right)_{\mathcal{A}}^{\up}\xrightarrow{\sim}\Uq^{-}\left(\mathfrak{g}\right)_{\mathcal{A}}^{\up}.
\]
\end{thm}
\begin{rem}
\textup{(1)} Theorem \ref{thm:dec} (1) can be shown directly in
finite type cases using the \PBW bases of $\Uq^{-}$ (see \cite[Proposition 5.3]{BG:double}).
Hence it is a new result only in infinite type cases.

\textup{(2)} For the proof of Theorem \ref{thm:dec} (1), we use
the dual canonical basis and the multiplication formula among them,
in particular we will prove Theorem \ref{thm:dec} (2).  After finishing
this work, the proof which does not involve the theory of the dual
canonical basis was informed to the author by Toshiyuki Tanisaki \cite[Proposition 2.10]{Tanisaki:KOY}.
He also proves the tensor product decomposition in Lusztig form, De
Concini-Kac form and De Concini-Procesi form.

We note that De Concini-Kac form (resp. De Concini-Procesi form) is
related with the dual integral form of Lusztig's integral form with
respect to the Kashiwara (resp. Lusztig) non-degenerate bilinear form
on $\mathbf{U}_{q}^{-}$ respectively. Since the multiplicative structure
among the dual canonical basis does not depend on a choice of the
non-degenerate bilinear form (and hence the definition of the dual
canonical basis), our argument yields results for the tensor product
decompositions of the De Concini-Kac form and De Concini-Procesi form.

\textup{(3)} We note that the fact that $\Uq^{-}\cap T_{w}\Uq^{\geq0}$
has a \PBW bases is known by Beck-Chari-Pressley \cite[Proposition 2.3]{BCP:affine}
in general. The injectivity in Theorem \ref{thm:dec} can be easily
proved by the linear independence of the \PBW monomials (see \cite[Theorem 40.2.1 (a)]{Lus:intro})
and the triangular decomposition of the quantized enveloping algebra
(see \cite[3.2]{Lus:intro}). Hence the non-trivial assertion is the
surjectivities in Theorem \ref{thm:dec}.
\end{rem}
\begin{NB}

For a Weyl group element $w\in W$, we consider the subalgebra which
is associated with $\Delta_{+}\left(>w\right):=\Delta_{+}\setminus\Delta_{+}\left(\leq w\right)=\Delta_{+}\cap w\Delta_{+}$
by $\Uq^{-}\cap T_{w}\Uq^{-}$. If $\mathfrak{g}$ is of finite type,
we have a $\Delta_{+}\setminus\Delta_{+}\left(\leq w\right)=\Delta_{+}\left(\leq ww_{0}\right)$
where $w_{0}$ is the longest element. Using PBW bases, it can be
shown that $\Uq^{-}\cap T_{w}\Uq^{-}=\Uq^{-}\cap T_{ww_{0}}\Uq^{\geq0}$.
In particular, $\Uq^{-}\cap T_{w}\Uq^{-}$ is compatible with the
dual canonical basis. The following result is the generalization for
arbitrary cases.

\end{NB}
\begin{thm}
\textup{(1)} For a Weyl group element $w\in W$ and a reduced word
$\bm{i}=\left(i_{1},\cdots,i_{\ell}\right)$ of $w$, we have 
\[
\Uq^{-}\cap T_{w}\Uq^{-}=\Uq^{-}\cap T_{i_{1}}\Uq^{-}\cap T_{i_{1}}T_{i_{2}}\Uq^{-}\cap\cdots\cap T_{i_{1}}\cdots T_{i_{\ell}}\Uq^{-}.
\]

\textup{(2)} $\Uq^{-}\cap T_{w}\Uq^{-}$ is compatible with the dual
canonical basis, that is $\mathbf{B}^{\up}\cap\Uq^{-}\cap T_{w}\Uq^{-}$
is a $\mathbb{Q}\left(q\right)$-basis of $\Uq^{-}\cap T_{w}\Uq^{-}$.
In fact there exists a subset $\mathscr{B}\left(\Uq^{-}\cap T_{w}\Uq^{-}\right)\subset\mathscr{B}\left(\infty\right)$
such that 
\[
\Uq^{-}\cap T_{w}\Uq^{-}=\bigoplus_{b\in\mathscr{B}\left(\Uq^{-}\cap T_{w}\Uq^{-}\right)}\mathbb{Q}\left(q\right)G^{\up}\left(b\right).
\]

\end{thm}
Using the theory of crystal basis, we can obtain the characterization
of the subset $\mathscr{B}\left(\Uq^{-}\cap T_{w}\Uq^{-}\right)$.

\subsubsection{}

For $w\in W$, we have the decomposition theorem of the crystal basis
$\mathscr{B}\left(\infty\right)$ of $\Uq^{-}$ associated with a
Weyl group element (and a reduced word) and the corresponding multiplication
formula. We consider the map $\Omega_{w}$ associated with a Weyl
group element which is introduced in Saito \cite{Saito:PBW} (and
Baumann-Kamnitzer-Tingley \cite{BKT:affine}) :
\[
\Omega_{w}:=\left(\tau_{\leq w},\tau_{>w}\right)\colon\mathscr{B}\left(\infty\right)\to\mathscr{B}\left(\Uq^{-}\cap T_{w}\Uq^{\geq0}\right)\times\mathscr{B}\left(\Uq^{-}\cap T_{w}\Uq^{-}\right),
\]
where $\tau_{\leq w}\left(b\right)$ and $\tau_{>w}\left(b\right)$
are defined by crystal basis as follows: 
\begin{align*}
L\left(b,\bm{i}\right) & :=\left(\varepsilon_{i_{1}}\left(b\right),\varepsilon_{i_{2}}\left(\hat{\sigma}_{i_{1}}^{*}b\right),\cdots,\varepsilon_{i_{\ell}}\left(\hat{\sigma}_{i_{\ell-1}}^{*}\cdots\hat{\sigma}_{i_{1}}^{*}b\right)\right)\in\mathbb{Z}_{\geq0}^{\ell},\\
b\left(\bm{c,\bm{i}}\right) & :=f_{i_{1}}^{\left(c_{1}\right)}T_{i_{1}}\left(f_{i_{2}}^{\left(c_{2}\right)}\right)\cdots T_{i_{1}}\cdots T_{i_{\ell-1}}\left(f_{i_{\ell}}^{\left(c_{\ell}\right)}\right)\;\mathrm{mod}\;q\mathscr{L}\left(\infty\right)\in\mathscr{B}\left(\infty\right),\\
\tau_{\leq w}\left(b\right) & :=b\left(L\left(b,\bm{i}\right),\bm{i}\right)\in\mathscr{B}\left(\infty\right),\\
\tau_{>\bm{i}}\left(b\right) & :=\sigma_{i_{1}}\cdots\sigma_{i_{\ell}}\hat{\sigma}_{i_{\ell}}^{*}\cdots\hat{\sigma}_{i_{1}}^{*}b\in\mathscr{B}\left(\infty\right).
\end{align*}

The following is the multiplicity-free result of the multiplication
among the dual canonical basis elements in finite part and co-finite
part.
\begin{thm}
Let $w$ be a Weyl group element and $\bm{i}=\left(i_{1},\cdots,i_{\ell}\right)$
be a reduced word of $w$.

For a crystal basis element $b\in\mathscr{B\left(\infty\right)}$,
we have 
\[
G^{\up}\left(\tau_{\leq w}\left(b\right)\right)G^{\up}\left(\tau_{>w}\left(b\right)\right)\in G^{\up}\left(b\right)+\sum_{L\left(b^{'},\bm{i}\right)<L\left(b,\bm{i}\right)}q\mathbb{Z}\left[q\right]G^{\up}\left(b^{'}\right)
\]
where $L\left(b^{'},\bm{i}\right)<L\left(b,\bm{i}\right)$ is the
left lexicographic order on $\mathbb{Z}_{\geq0}^{\ell}$ associated
with a reduced word $\bm{i}$.
\end{thm}
Using the induction on the lexicographic order on each root space,
we obtain the surjectivity in Theorem \ref{thm:dec} (2). In particular,
\ref{thm:dec} (1) can be shown.

Since the subalgebras $\Uq^{-}\cap T_{w}\Uq^{\geq0}$ and $\Uq^{-}\cap T_{w}\Uq^{-}$
are compatible with the dual canonical base and the dual bar-involution
$\sigma$ which characterizes the dual canonical base is an (twisted)
anti-involution, we obtain the tensor product factorization in the
opposite order.
\begin{cor}
For a Weyl group element $w\in W$, multiplication in $\Uq^{-}$ defines
an isomorphism of vector spaces:
\begin{align*}
\left(\Uq^{-}\cap T_{w}\Uq^{-}\right)\otimes\left(\Uq^{-}\cap T_{w}\Uq^{\geq0}\right) & \xrightarrow{\sim}\Uq^{-}.
\end{align*}
\end{cor}
\begin{acknowledgement*}
This work was done while the author was visiting the Institut de Recherche
Mathematique Avancee (Strasbourg) with support by the JSPS Program
for Advancing Strategic International Networks to Accelerate the Circulation
of Talented Researchers “Mathematical Science of Symmetry, Topology
and Moduli, Evolution of International Research Network based on OCAMI”.
The author thanks Toshiyuki Tanisaki on the fruitful discussions on
the topic. The author is also grateful to Pierre Baumann and Fan Qin
for the discussions along the stay in Strasbourg.
\end{acknowledgement*}

\section{Review on Quantum unipotent subgroup and Dual canonical basis}

\subsection{Quantum universal enveloping algebra}

In this subsection, we give a brief review of the definition of quantum
universal enveloping algebra.

\subsubsection{}

Let $I$ be a finite index set.
\begin{defn}
A \emph{root datum} is a quintuple $\left(A,P,\Pi,P^{\vee},\Pi^{\vee}\right)$
which consists of 
\begin{enumerate}
\item a square matrix $\left(a_{ij}\right)_{i,j\in I}$, called the symmetrizable
generalized Cartan matrix, that is an $I$-indexed $\mathbb{Z}$-valued
matrix which satisfies

\begin{enumerate}
\item $a_{ii}=2$ for $i\in I$
\item $a_{ij}\in\mathbb{Z}_{\leq0}$ for $i\neq j$
\item there exists a diagonal matrix $\mathrm{diag}\left(d_{i}\right)_{i\in I}$
such that $\left(d_{i}a_{ij}\right)_{i,j\in I}$ is symmetric and
$d_{i}$ are positive integers.
\end{enumerate}
\item $P$ : a free abelian group (the weight lattice),
\item $\Pi=\left\{ \alpha_{i}\mid i\in I\right\} \subset P$ : the set of
simple roots such that $\Pi\subset P\otimes_{\mathbb{Z}}\mathbb{Q}$
is linearly independent,
\item $P^{\vee}=\mathrm{Hom}_{\mathbb{Z}}\left(P,\mathbb{Z}\right)$ : the
dual lattice (the coweight lattice) of $P$ with perfect paring $\left\langle \cdot,\cdot\right\rangle \colon P^{\vee}\otimes_{\mathbb{Z}}P\to\mathbb{Z}$,
\item $\Pi^{\vee}=\left\{ h_{i}\mid i\in I\right\} \subset P^{\vee}$ :
the set of simple coroots, satisfying the following properties:

\begin{enumerate}
\item $a_{ij}=\left\langle h_{i},\alpha_{j}\right\rangle $ for all $i,j\in I$,
\item There exists $\left\{ \Lambda_{i}\right\} _{i\in I}\subset P$, called
the set of fundamental weights, satisfying $\left\langle h_{i},\Lambda_{j}\right\rangle =\delta_{ij}$
for $i,j\in I$.
\end{enumerate}
\end{enumerate}
\end{defn}
We say $\Lambda\in P$ is \emph{dominant} if $\left\langle h_{i},\Lambda\right\rangle \geq0$
for any $i\in I$ and denote by $P_{+}$ the set of dominant integral
weights. Let $Q=\bigoplus_{i\in I}\mathbb{Z}\alpha_{i}\subset P$
be the root lattice. Let $Q_{\pm}=\pm\sum_{i\in I}\mathbb{Z}_{\geq0}\alpha_{i}$.
For $\xi=\sum_{i\in I}\xi_{i}\alpha_{i}\in Q$, we set $|\xi|=\sum_{i\in I}\xi_{i}$.

\subsubsection{}

Let $\left(A,P,\Pi,P^{\vee},\Pi^{\vee}\right)$ be a root datum. We
set $\mathfrak{h}:=P\otimes_{\mathbb{Z}}\mathbb{C}$. A triple $\left(\mathfrak{h},\Pi,\Pi^{\vee}\right)$
is called a Cartan datum or a realization of a generalized Cartan
matrix $A$.

It is known that there exists a symmetric bilinear form $\left(\cdot,\cdot\right)$
on $\mathfrak{h}^{*}$ satisfying
\begin{enumerate}
\item $\left(\alpha_{i},\alpha_{i}\right)=d_{i}a_{ij}$,
\item $\left\langle h_{i},\lambda\right\rangle =2\left(\alpha_{i},\lambda\right)/\left(\alpha_{i},\alpha_{i}\right)$
for $i\in I$ and $\lambda\in\mathfrak{h}^{*}$.\end{enumerate}
\begin{defn}
Let $\mathfrak{g}$ be the \emph{symmetrizable Kac-Moody Lie algebra}
associated with a realization $\left(\mathfrak{h},\Pi,\Pi^{\vee}\right)$
of a symmetrizable generalized Cartan matrix $A=\left(a_{ij}\right)_{i,j\in I}$,
that is a Lie algebra which is generated by $\left\{ e_{i}\right\} _{i\in I}\cup\left\{ f_{i}\right\} _{i\in I}\cup\mathfrak{h}$
with the following relations:
\begin{enumerate}
\item $\left[h_{1},h_{2}\right]=0$ for $h_{1},h_{2}\in\mathfrak{h}$,
\item $\left[h,e_{i}\right]=\left\langle h,\alpha_{i}\right\rangle e_{i}$,
$\left[h,f_{i}\right]=-\left\langle h,\alpha_{i}\right\rangle f_{i}$
for $h\in\mathfrak{h}$ and $i\in I$,
\item $\left[e_{i},f_{j}\right]=\delta_{ij}\alpha_{i}^{\vee}$ for $i,j\in I$,
\item $\mathrm{ad}\left(e_{i}\right)^{1-a_{ij}}\left(e_{j}\right)=0$ and
$\mathrm{ad}\left(f_{i}\right)^{1-a_{ij}}\left(f_{j}\right)=0$ for
$i\neq j$, where $\mathrm{ad}\left(x\right)\left(y\right)=\left[x,y\right]$.
\end{enumerate}
\end{defn}
Let $\mathfrak{n}_{+}$ (resp.\ $\mathfrak{n}_{-}$) be the Lie subalgebra
which is generated by $\left\{ e_{i}\right\} _{i\in I}$ (resp. $\left\{ f_{i}\right\} _{i\in I}$).
We have the triangular decomposition and the root space decomposition
\[
\mathfrak{g}=\mathfrak{n}_{-}\oplus\mathfrak{h}\oplus\mathfrak{n}_{+}=\mathfrak{h}\oplus\bigoplus_{\alpha\in\mathfrak{h}^{*}\setminus\left\{ 0\right\} }\mathfrak{g}_{\alpha}
\]
where $\mathfrak{g}_{\alpha}=\left\{ x\in\mathfrak{g}\mid\left[h,x\right]=\left\langle h,\alpha\right\rangle x\;{}^{\forall}h\in\mathfrak{h}\right\} $.
The set $\Delta:=\left\{ \alpha\in\mathfrak{h}^{*}\setminus\left\{ 0\right\} \mid\mathfrak{g}_{\alpha}\neq0\right\} $
is called the root system of $\mathfrak{g}$.

\subsubsection{}

We fix a root datum $\left(A,P,\Pi,P^{\vee},\Pi^{\vee}\right)$. We
introduce an indeterminate $q$. For $i\in I$, we set $q_{i}=q^{d_{i}}$.
For $\xi=\sum\xi_{i}\alpha_{i}\in Q$, we set $q_{\xi}:=\prod_{i\in I}q_{i}^{\xi_{i}}$.

For $n\in\mathbb{Z}$ and $i\in I$, we set. 
\[
\left[n\right]_{i}:=\dfrac{q_{i}^{n}-q_{i}^{-n}}{q_{i}-q_{i}^{-1}}
\]
and $[n]_{i}!=[n]_{i}[n-1]_{i}\cdots[1]_{i}$ for $n>0$ and $[0]!=1$.
\begin{defn}
The \emph{quantized enveloping algebra $\Uq\left(\mathfrak{g}\right)$
}associated with a root datum $\left(A,P,\Pi,P^{\vee},\Pi^{\vee}\right)$
is the $\mathbb{Q}\left(q\right)$-algebra which is generated by $\left\{ e_{i}\right\} _{i\in I}$,
$\left\{ f_{i}\right\} _{i\in I}$ and $\left\{ q^{h}\mid h\in P^{\vee}\right\} $
with the following relations:
\begin{enumerate}
\item $q^{0}=1$ and $q^{h+h^{'}}=q^{h}q^{h'}$ for $h,h^{'}\in P^{\vee}$,
\item $q^{h}e_{i}q^{-h}=q^{\left\langle h,\alpha_{i}\right\rangle }e_{i}$,
$q^{h}f_{i}q^{-h}=q^{-\left\langle h,\alpha_{i}\right\rangle }f_{i}$
for $i\in I$ and $h\in P^{\vee}$,
\item $e_{i}f_{j}-f_{j}e_{i}=\delta_{ij}\left(k_{i}-k_{i}^{-1}\right)/\left(q_{i}-q_{i}^{-1}\right)$
where $k_{i}=q^{d_{i}h_{i}}$,
\item ${\displaystyle \sum_{k=0}^{1-a_{ij}}(-1)^{k}e_{i}^{(1-a_{ij}-k)}e_{j}e_{i}^{(k)}=\sum_{k=0}^{1-a_{ij}}(-1)^{k}f_{i}^{(1-a_{ij}-k)}f_{j}f_{i}^{(k)}=0}$~($q$-Serre
relations),
\end{enumerate}
where $e_{i}^{(k)}=e_{i}^{k}/[k]_{i}!$, $f_{i}^{(k)}=f_{i}^{k}/[k]_{i}!$
for $i\in I$ and $k\in\mathbb{Z}_{>0}$. 
\end{defn}

\subsubsection{}

Let $\Uq^{0}$ be the subalgebra which is generated by $\left\{ q^{h}\mid h\in P^{\vee}\right\} $,
it is isomorphic to the group algebra $\mathbb{Q}\left(q\right)\left[P^{\vee}\right]:=\bigoplus_{h\in P^{\vee}}\mathbb{Q}\left(q\right)q^{h}$
over $\mathbb{Q}\left(q\right)$. For $\xi=\sum_{i\in I}\xi_{i}\alpha_{i}\in Q$,
we set $k_{\xi}:=\prod_{i\in I}k_{i}^{\xi_{i}}=\prod_{i\in I}q^{d_{i}\xi_{i}h_{i}}$. 

Let $\mathbf{U}_{q}^{+}$ be the $\mathbb{Q}\left(q\right)$-subalgebra
generated by $\left\{ e_{i}\right\} _{i\in I}$, $\mathbf{U}_{q}^{-}$
be the $\mathbb{Q}\left(q\right)$-subalgebra generated by $\left\{ f_{i}\right\} _{i\in I}$,
$\mathbf{U}_{q}^{\geq0}$ be the $\mathbb{Q}\left(q\right)$-subalgebra
generated by $\Uq^{0}$ and $\mathbf{U}_{q}^{+}$, and $\Uq^{\leq0}$
be the $\mathbb{Q}\left(q\right)$-subalgebra generated by $\Uq^{0}$
and $\mathbf{U}_{q}^{-}$.
\begin{thm}[{\cite[Corollary 3.2.5]{Lus:intro}}]
The multiplication of $\Uq$ induces the triangular decomposition
of $\Uq\left(\mathfrak{g}\right)$ as vector spaces over $\mathbb{Q}\left(q\right)$:
\begin{equation}
\Uq\left(\mathfrak{g}\right)\cong\Uq^{+}\otimes\Uq^{0}\otimes\Uq^{-}\cong\Uq^{-}\otimes\Uq^{0}\otimes\Uq^{+}.
\end{equation}

\end{thm}

\subsubsection{}

For $\xi\in\pm Q$, we define its root space $\Uq^{\pm}\left(\mathfrak{g}\right)_{\xi}$
by 
\begin{equation}
\Uq^{\pm}\left(\mathfrak{g}\right)_{\xi}:=\left\{ x\in\Uq^{\pm}\left(\mathfrak{g}\right)\middle|q^{h}xq^{-h}=q^{\left\langle h,\xi\right\rangle }x\;\text{for}\;h\in P^{\vee}\right\} .
\end{equation}
Then we have a root space decomposition $\Uq^{\pm}\left(\mathfrak{g}\right)=\bigoplus_{\xi\in Q_{\pm}}\Uq^{\pm}\left(\mathfrak{g}\right)_{\xi}$. 

An element $x\in\Uq^{\pm}\left(\mathfrak{g}\right)$ is called homogenous
if $x\in\Uq^{\pm}\left(\mathfrak{g}\right)_{\xi}$ for some $\xi\in Q_{\pm}$.

\subsubsection{}

We define a $\mathbb{Q}\left(q\right)$-algebra anti-involution $*\colon\Uq\left(\mathfrak{g}\right)\to\Uq\left(\mathfrak{g}\right)$
by
\begin{align}
*(e_{i})=e_{i}, &  & *(f_{i})=f_{i}, &  & *(q^{h})=q^{-h}.\label{eq:star}
\end{align}

We call this \emph{star involution}.

We define a $\mathbb{Q}$-algebra automorphism $\overline{\phantom{x}}\colon\Uq\left(\mathfrak{g}\right)\to\Uq\left(\mathfrak{g}\right)$
by 
\begin{align}
\overline{e_{i}}=e_{i}, &  & \overline{f_{i}}=f_{i}, &  & \overline{q}=q^{-1}, &  & \overline{q^{h}}=q^{-h}.\label{eq:bar}
\end{align}
We call this the \emph{bar involution}.

We remark that these two involutions preserve $\Uq^{+}(\mathfrak{g})$
and $\Uq^{-}(\mathfrak{g})$, and we have $\overline{\phantom{x}}\circ*=*\circ\overline{\phantom{x}}$.

\subsubsection{}

In this article, we choose the following comultiplication $\Delta=\Delta_{-}$
on $\Uq\left(\mathfrak{g}\right)$:
\begin{align}
\Delta(q^{h})=q^{h}\otimes q^{h} & , & \Delta(e_{i})=e_{i}\otimes k_{i}^{-1}+1\otimes e_{i} & , & \Delta(f_{i})=f_{i}\otimes1+k_{i}\otimes f_{i}.
\end{align}

\subsubsection{}

We have a symmetric non-degenerate bilinear form $\left(\cdot,\cdot\right)=\left(\cdot,\cdot\right)_{-}\colon\Uq^{-}\otimes\Uq^{-}\to\mathbb{Q}\left(q\right)$
. We define a $\mathbb{Q}\left(q\right)$-algebra structure on $\Uq^{-}\otimes\Uq^{-}$
by 
\begin{equation}
(x_{1}\otimes y_{1})(x_{2}\otimes y_{2})=q^{-(\wt(x_{2}),\wt(y_{1}))}x_{1}x_{2}\otimes y_{1}y_{2},
\end{equation}
where $x_{i},y_{i}~(i=1,2)$ are homogenous elements.

Let $r=r_{-}\colon\Uq^{-}\to\Uq^{-}\otimes\Uq^{-}$ be the $\mathbb{Q}\left(q\right)$-algebra
homomorphism defined by 
\[
r\left(f_{i}\right)=f_{i}\otimes1+1\otimes f_{i}\;\left(i\in I\right).
\]
We call this the twisted comultiplication.

Then it is known that there exists a unique $\mathbb{Q}\left(q\right)$-valued
non-degenerate symmetric bilinear form $\left(\cdot,\cdot\right)\colon\Uq^{-}\otimes\Uq^{-}\to\mathbb{Q}\left(q\right)$
with the following properties:
\begin{align*}
\left(1,1\right)=1, & \left(f_{i},f_{j}\right)=\delta_{ij}, & \left(r\left(x\right),y_{1}\otimes y_{2}\right)=\left(x,y_{1}y_{2}\right), & \left(x_{1}\otimes x_{2},r\left(y\right)\right)=\left(x_{1}x_{2},y\right)
\end{align*}
for homogenous $x,y_{1},y_{2}\in\mathbf{U}_{q}^{-}$ where the form
$\left(\cdot\otimes\cdot,\cdot\otimes\cdot\right)\colon\left(\Uq^{-}\otimes\Uq^{-}\right)\otimes\left(\Uq^{-}\otimes\Uq^{-}\right)\to\mathbb{Q}\left(q\right)$
is defined by $\left(x_{1}\otimes x_{2},y_{1}\otimes y_{2}\right)=\left(x_{1},y_{1}\right)\left(x_{2}\otimes y_{2}\right)$
for $x_{1},x_{2},y_{1},y_{2}\in\Uq^{-}$.

\subsubsection{}

For $i\in I$, we define the unique $\mathbb{Q}\left(q\right)$-linear
map $_{i}r\colon\Uq^{-}\to\Uq^{-}$ (resp.\ $r_{i}\colon\Uq^{-}\to\Uq^{-}$)
defined by 
\begin{align*}
\left(_{i}r\left(x\right),y\right) & =\left(x,f_{i}y\right),\\
\left(r_{i}\left(x\right),y\right) & =\left(x,yf_{i}\right).
\end{align*}

\begin{lem}
For $x,y\in\mathbf{U}_{q}^{-}$, we have $q$-Boson relations:
\begin{align*}
_{i}r\left(xy\right) & ={}_{i}r\left(x\right)y+q^{\left(\wt x,\alpha_{i}\right)}x{}_{i}r\left(y\right),\\
r_{i}\left(xy\right) & =q^{\left(\wt y,\alpha_{i}\right)}r_{i}\left(x\right)y+xr_{i}\left(y\right).
\end{align*}

\end{lem}
\begin{NB}
\begin{proof}
Let $r\left(z\right)=z_{\left(1\right)}\otimes z_{\left(2\right)}$
be in Sweedler notation. Then we have 
\begin{align*}
r\left(f_{i}z\right) & =\left(f_{i}\otimes1+1\otimes f_{i}\right)\left(z_{\left(1\right)}\otimes z_{\left(2\right)}\right)\\
 & =f_{i}z_{\left(1\right)}\otimes z_{\left(2\right)}+q^{-\left(-\alpha_{i},\wt z_{\left(1\right)}\right)}z_{\left(1\right)}\otimes f_{i}z_{\left(2\right)}\\
 & =f_{i}z_{\left(1\right)}\otimes z_{\left(2\right)}+q^{\left(\alpha_{i},\wt z_{\left(1\right)}\right)}z_{\left(1\right)}\otimes f_{i}z_{\left(2\right)}.
\end{align*}

Then we obtain 
\begin{align*}
\left(_{i}r\left(xy\right),z\right) & =\left(xy,f_{i}z\right)=\left(x\otimes y,r\left(f_{i}z\right)\right)\\
 & =\left(x\otimes y,f_{i}z_{\left(1\right)}\otimes z_{\left(2\right)}+q^{\left(\alpha_{i},\wt x\right)}z_{\left(1\right)}\otimes f_{i}z_{\left(2\right)}\right)\\
 & =\left(_{i}r\left(x\right)\otimes y,z_{\left(1\right)}\otimes z_{\left(2\right)}\right)+q^{\left(\alpha_{i},\wt x\right)}\left(x\otimes{}_{i}r\left(y\right),z_{\left(1\right)}\otimes z_{\left(2\right)}\right)\\
 & =\left(_{i}r\left(x\right)\cdot y+q^{\left(\alpha_{i},\wt x\right)}x\cdot{}_{i}r\left(y\right),z\right).
\end{align*}

The second equation can be shown in the similar manner.
\end{proof}
\end{NB}
\begin{lem}
We have 
\begin{equation}
\left[e_{i},x\right]=\dfrac{r_{i}\left(x\right)k_{i}-k_{i}^{-1}{}_{i}r\left(x\right)}{q_{i}-q_{i}^{-1}}
\end{equation}
for $x\in\mathbf{U}_{q}^{-}$.
\end{lem}
\begin{NB}
\begin{proof}
We prove by the induction on $Q$-grading in $\mathbf{U}_{q}^{-}$.
For $x=f_{j}\;\left(j\in I\right)$, we have the claim by $_{i}r\left(f_{j}\right)=r_{i}\left(f_{j}\right)=\delta_{ij}$.
For $x=x'x''$, we have 
\begin{align*}
\left[e_{i},x'x''\right] & =\left[e_{i},x'\right]x''+x'\left[e_{i},x''\right]\\
 & =\dfrac{r_{i}\left(x'\right)k_{i}-k_{i}^{-1}{}_{i}r\left(x'\right)}{q_{i}-q_{i}^{-1}}x''+x'\dfrac{r_{i}\left(x''\right)k_{i}-k_{i}^{-1}{}_{i}r\left(x''\right)}{q_{i}-q_{i}^{-1}}\\
 & =\dfrac{q^{\left(\alpha_{i},\wt x''\right)}r_{i}\left(x'\right)x''k_{i}-k_{i}^{-1}{}_{i}r\left(x'\right)x''}{q_{i}-q_{i}^{-1}}+\dfrac{x'r_{i}\left(x''\right)k_{i}-q^{\left(\alpha_{i},\wt\left(x'\right)\right)}k_{i}^{-1}x'{}_{i}r\left(x''\right)}{q_{i}-q_{i}^{-1}}\\
 & =\frac{\left(q^{\left(\alpha_{i},\wt x''\right)}r_{i}\left(x'\right)x''+x'r_{i}\left(x''\right)\right)k_{i}-k_{i}^{-1}\left(_{i}r\left(x'\right)x''+q^{\left(\alpha_{i},\wt\left(x'\right)\right)}x'{}_{i}r\left(x''\right)\right)}{q_{i}-q_{i}^{-1}}\\
 & =\frac{r_{i}\left(x\right)k_{i}-k_{i}^{-1}{}_{i}r\left(x\right)}{q_{i}-q_{i}^{-1}}.
\end{align*}

Hence we obtain the claim.
\end{proof}
\end{NB}

Using the $q$-Boson relation, we obtain the following proposition.
\begin{lem}
\label{lem:crydecomp}For each $i\in I$, any element $x\in\Uq^{-}$
can be written uniquely as 
\[
x=\sum_{c\geq0}f_{i}^{\left(c\right)}x_{c}\;\text{with}\;x_{c}\in\mathrm{Ker}\left(_{i}r\right).
\]

\end{lem}

\subsection{Canonical basis and dual canonical basis}

In this subsection, we give a brief review of the theory of the canonical
basis and the dual canonical basis following Kashiwara. Note that
Kashiwara called it the lower global basis and the upper global basis.

\subsubsection{}

We define $\mathbb{Q}$-subalgebras $\mathcal{A}_{0}$, $\mathcal{A}_{\infty}$
and $\mathcal{A}$ of $\mathbb{Q}\left(q\right)$ by
\begin{align*}
\mathcal{A}_{0} & :=\{f\in\mathbb{Q}\left(q\right);f\text{~is regular at~}q=0\},\\
\mathcal{A}_{\infty} & :=\{f\in\mathbb{Q}\left(q\right);f\text{~is regular at~}q=\infty\},\\
\mathcal{A} & :=\mathbb{Q}[q^{\pm1}].
\end{align*}

\subsubsection{}

We define the Kashiwara operator $\tilde{e}_{i}$, $\tilde{f}_{i}$
on $\Uq^{-}$ by
\begin{align*}
\tilde{e}_{i}x & =\sum_{c\geq1}f_{i}^{\left(c-1\right)}x_{c},\\
\tilde{f}_{i}x & =\sum_{c\geq0}f_{i}^{\left(c+1\right)}x_{c},
\end{align*}
and we set 
\begin{align*}
\mathscr{L}\left(\infty\right) & :=\sum_{\substack{\ell\geq0\\
i_{1},\cdots,i_{\ell}\in i
}
}\mathcal{A}_{0}\tilde{f}_{i_{1}}\cdots\tilde{f}_{i_{\ell}}1\subset\Uq^{-},\\
\mathscr{B}\left(\infty\right) & :=\left\{ \tilde{f}_{i_{1}}\cdots\tilde{f}_{i_{\ell}}1\;\mathrm{mod}\;q\mathscr{L}\left(\infty\right)\middle|l\geq0,i_{1},\cdots,i_{\ell}\in I\right\} \subset\mathscr{L}\left(\infty\right)/q\mathscr{L}\left(\infty\right).
\end{align*}

Then $\mathscr{L}\left(\infty\right)$ is a $\mathcal{A}_{0}$-lattice
with $\mathbb{Q}\left(q\right)\otimes_{\mathcal{A}_{0}}\mathscr{L}\left(\infty\right)\simeq\Uq^{-}$
and stable under $\tilde{e}_{i}$ and $\tilde{f}_{i}$ . $\mathscr{B}\left(\infty\right)$
is a $\mathbb{Q}$-basis of $\mathscr{L}\left(\infty\right)/q\mathscr{L}\left(\infty\right)$.
We also have induced maps $\tilde{f}_{i}\colon\mathscr{B}\left(\infty\right)\to\mathscr{B}\left(\infty\right)$
and $\tilde{e}_{i}\colon\mathscr{B}\left(\infty\right)\to\mathscr{B}\left(\infty\right)\sqcup\left\{ 0\right\} $
with the property that $\tilde{f}_{i}\tilde{e}_{i}b=b$ for $b\in\mathscr{B}\left(\infty\right)$
with $\tilde{e}_{i}b\neq0$. We call $\left(\mathscr{B}\left(\infty\right),\mathscr{L}\left(\infty\right)\right)$
the (lower) crystal basis of $\Uq^{-}$ and call $\mathscr{L}\left(\infty\right)$
the (lower) crystal lattice. We denote $1\;\mathrm{mod}\;q\mathscr{L}\left(\infty\right)$
by $u_{\infty}$.

\subsubsection{}

It is also known that $*\colon\Uq^{-}\to\Uq^{-}$ induces $*\colon\mathscr{L}\left(\infty\right)\to\mathscr{L}\left(\infty\right)$
and $*\colon\mathscr{B}\left(\infty\right)\to\mathscr{B}\left(\infty\right)$.
We setand $\tilde{f}_{i}^{*}:=*\circ\tilde{f}_{i}\circ*\colon\mathscr{B}\left(\infty\right)\to\mathscr{B}\left(\infty\right)$
and $\tilde{e}_{i}^{*}:=*\circ\tilde{e}_{i}\circ*\colon\mathscr{B}\left(\infty\right)\to\mathscr{B}\left(\infty\right)\sqcup\left\{ 0\right\} $

\subsubsection{}

Let $\overline{\mathscr{L}\left(\infty\right)}=\left\{ \overline{x}\mid x\in\mathscr{L}\left(\infty\right)\right\} $.
Then the natural map
\[
\mathscr{L}\left(\infty\right)\cap\overline{\mathscr{L}\left(\infty\right)}\cap\Uq^{-}\left(\mathfrak{g}\right)_{\mathcal{A}}\to\mathscr{L}\left(\infty\right)/q\mathscr{L}\left(\infty\right)
\]
is an isomorphism of $\mathbb{Q}$-vector spaces, let $G^{\low}$
be the inverse of this isomorphism. The image
\[
\mathbf{B}^{\low}=\left\{ G^{\low}\left(b\right)\mid b\in\mathscr{B}\left(\infty\right)\right\} \subset\mathscr{L}\left(\infty\right)\cap\overline{\mathscr{L}\left(\infty\right)}\cap\Uq^{-}\left(\mathfrak{g}\right)_{\mathcal{A}}
\]
is an $\mathcal{A}$-basis of $\Uq^{-}\left(\mathfrak{g}\right)_{\mathcal{A}}$
and is called the canonical basis or the lower global basis of $\Uq^{-}$.

\subsubsection{}

The important property of the canonical basis is the following compatibility
with the left and right ideal which are generated by Chevalley generators
$\left\{ f_{i}\right\} _{i\in I}$.
\begin{thm}[{\cite[Theorem 14.3.2, Theorem 14.4.3]{Lus:intro},\cite[Theorem 7]{Kas:crystal}}]
For $i\in I$ and $n\ge1$, $f_{i}^{n}\Uq^{-}$ and $\Uq^{-}f_{i}^{n}$
is compatible with the canonical base, that is $f_{i}^{n}\Uq^{-}\cap\mathbf{B}^{\low}$
(resp. $\Uq^{-}f_{i}^{n}\cap\mathbf{B}^{\low}$) is a basis of $f_{i}^{n}\Uq^{-}$
(resp. $\Uq^{-}f_{i}^{n}$). In fact, we have 
\begin{align*}
f_{i}^{n}\Uq^{-}\cap\Uq^{-}\left(\mathfrak{g}\right)_{\mathcal{A}} & =\bigoplus_{\substack{b\in\mathscr{B}\left(\infty\right)\\
\varepsilon_{i}\left(b\right)\geq m
}
}\mathcal{A}G^{\low}\left(b\right),\\
\Uq^{-}f_{i}^{n}\cap\Uq^{-}\left(\mathfrak{g}\right)_{\mathcal{A}} & =\bigoplus_{\substack{b\in\mathscr{B}\left(\infty\right)\\
\varepsilon_{i}^{*}\left(b\right)\geq m
}
}\mathcal{A}G^{\low}\left(b\right)
\end{align*}

\end{thm}

\subsubsection{}

Let $\sigma\colon\Uq^{-}\to\Uq^{-}$ be the $\mathbb{Q}$-linear map
defined by 
\[
\left(\sigma\left(x\right),y\right)=\overline{\left(x,\overline{y}\right)}
\]
for arbitrary $x,y\in\Uq^{-}$. Let $\sigma\left(\mathscr{L}\left(\infty\right)\right):=\left\{ \sigma\left(x\right)\mid x\in\mathscr{L}\left(\infty\right)\right\} $
and set the dual integral form: 
\[
\Uq^{-}\left(\mathfrak{g}\right)_{\mathcal{A}}^{\up}:=\left\{ x\in\Uq^{-}\mid\left(x,\Uq^{-}\left(\mathfrak{g}\right)_{\mathcal{A}}\right)\subset\mathcal{A}\right\} .
\]
$\Uq^{-}\left(\mathfrak{g}\right)_{\mathcal{A}}^{\up}$ has an $\mathcal{A}$-subalgebra
of $\Uq^{-}$. The natural map 
\[
\mathscr{L}\left(\infty\right)\cap\sigma\left(\mathscr{L}\left(\infty\right)\right)\cap\Uq^{-}\left(\mathfrak{g}\right)_{\mathcal{A}}^{\up}\to\mathscr{L}\left(\infty\right)/q\mathscr{L}\left(\infty\right)
\]
is also an isomorphism of $\mathbb{Q}$-vector spaces, so let $G^{\up}$
be the inverse of the above isomorphism.
\[
\mathbf{B}^{\up}=\left\{ G^{\up}\left(b\right)\mid b\in\mathscr{B}\left(\infty\right)\right\} \subset\mathscr{L}\left(\infty\right)\cap\sigma\left(\mathscr{L}\left(\infty\right)\right)\cap\Uq^{-}\left(\mathfrak{g}\right)_{\mathcal{A}}^{\up}
\]
is an $\mathcal{A}$-basis of $\Uq^{-}\left(\mathfrak{g}\right)_{\mathcal{A}}^{\up}$
and is called the dual canonical basis or the upper global basis of
$\Uq^{-}$.
\begin{rem}
We note that this definition of the dual canonical basis $\mathbf{B}^{\up}$
does depend on a choice of a non-degenerate bilinear form on $\mathbf{U}_{q}^{-}\left(\mathfrak{g}\right)$. 

\begin{NB}Although the definition of the dual canonical basis involves
the non-degenerate bilinear form, it can be checked that the structure
constants with respect to the dual canonical basis does not depend
on a choice. 

\end{NB}\end{rem}
\begin{prop}[{\cite[Proposition 4.26 (1)]{Kim:qunip}}]
\label{prop:dualsimpleroot}

For $i\in I$ and $c\geq1$, let 
\[
f_{i}^{\left\{ c\right\} }=f_{i}^{\left(c\right)}/\left(f_{i}^{\left(c\right)},f_{i}^{\left(c\right)}\right),
\]
then we have $f_{i}^{\left\{ c\right\} }=q_{i}^{c\left(c-1\right)/2}\left(f_{i}/\left(f_{i},f_{i}\right)\right)^{c}=q_{i}^{c\left(c-1\right)/2}f_{i}^{c}\in\mathbf{B}^{\up}$.
\end{prop}
We note that we have used the normalization $\left(f_{i},f_{i}\right)=1$
in the above proposition.

\subsubsection{}

For the dual canonical basis, we have the following expansion of left
and right multiplication with respect to the Chevalley generators
and its (shifted) powers.
\begin{thm}[{\cite[Proposition 2.2]{Kas:param}, \cite[Proposition 4.14 (ii)]{Oya:qfa}}]
\label{thm:multdualChevalley}

For $b\in\mathscr{B}\left(\infty\right)$, $i\in I$ and $c\geq1$,
we have \begin{subequations} 
\begin{align}
f_{i}^{\left\{ c\right\} }G^{\up}\left(b\right) & =q_{i}^{-c\varepsilon_{i}\left(b\right)}G^{\up}\left(\tilde{f}_{i}^{c}b\right)+\sum_{\varepsilon_{i}\left(b^{'}\right)<\varepsilon_{i}\left(b\right)+c}F_{i;b,b^{'}}^{\left\{ c\right\} }\left(q\right)G^{\up}\left(b^{'}\right),\\
G^{\up}\left(b\right)f_{i}^{\left\{ c\right\} } & =q_{i}^{-c\varepsilon_{i}^{*}\left(b\right)}G^{\up}\left(\tilde{f}_{i}^{*c}b\right)+\sum_{\varepsilon_{i}^{*}\left(b^{'}\right)<\varepsilon_{i}^{*}\left(b\right)+c}F_{i;b,b^{'}}^{*\left\{ c\right\} }\left(q\right)G^{\up}\left(b^{'}\right),
\end{align}
\end{subequations}where 
\begin{align*}
F_{i;b,b^{'}}^{\left\{ c\right\} }\left(q\right):= & \left(f_{i}^{\left\{ c\right\} }G^{\up}\left(b\right),G^{\low}\left(b^{'}\right)\right)=q_{i}^{c\left(c-1\right)/2}\left(G^{\up}\left(b\right),\left(_{i}r\right)^{c}G^{\low}\left(b^{'}\right)\right)\in q_{i}^{-c\varepsilon_{i}\left(b\right)}q\mathbb{Z}\left[q\right],\\
F_{i;b,b^{'}}^{*\left\{ c\right\} }\left(q\right):= & \left(G^{\up}\left(b\right)f_{i}^{\left\{ c\right\} },G^{\low}\left(b^{'}\right)\right)=q_{i}^{c\left(c-1\right)/2}\left(G^{\up}\left(b\right),\left(r_{i}\right)^{c}G^{\low}\left(b^{'}\right)\right)\in q_{i}^{-c\varepsilon_{i}^{*}\left(b\right)}q\mathbb{Z}\left[q\right].
\end{align*}

\end{thm}

\subsection{Braid group action and the (dual) canonical basis}

In this subsection, we recall the compatibility between Lusztig's
braid symmetry and the (dual) canonical basis (for more details, see
\cite[Section 4.4, Section 4,6]{Kim:qunip})

\subsubsection{Braid group action on quantized enveloping algebra}

Following Lusztig \cite[Section 37.1.3]{Lus:intro}, we define the
$\mathbb{Q}\left(q\right)$-algebra automorphisms $T'_{i,\epsilon}\colon\Uq\left(\mathfrak{g}\right)\to\Uq\left(\mathfrak{g}\right)$
and $T_{i,\epsilon}''\colon\Uq\left(\mathfrak{g}\right)\to\Uq\left(\mathfrak{g}\right)$
for $i\in I$ and $\epsilon\in\left\{ \pm1\right\} $ by the following
formulae:\begin{subequations} 
\begin{align}
T'_{i,\epsilon}\left(q^{h}\right) & =q^{s_{i}\left(h\right)},\\
T'_{i,\epsilon}\left(e_{j}\right) & =\begin{cases}
-k_{i}^{\epsilon}e_{i} & \text{for}\;j=i,\\
{\displaystyle \sum_{r+s=-\left\langle h_{i},\alpha_{j}\right\rangle }\left(-1\right)^{r}q_{i}^{\epsilon r}e_{i}^{\left(r\right)}e_{j}e_{i}^{\left(r\right)}} & \text{for}\;j\neq i,
\end{cases}\\
T'_{i,\epsilon}\left(f_{j}\right) & =\begin{cases}
-e_{i}k_{i}^{-\epsilon} & \text{for}\;j=i,\\
{\displaystyle \sum_{r+s=-\left\langle h_{i},\alpha_{j}\right\rangle }\left(-1\right)^{r}q_{i}^{-\epsilon r}f_{i}^{\left(s\right)}f_{j}f_{i}^{\left(r\right)}} & \text{for}\;j\neq i,
\end{cases}
\end{align}

\end{subequations} and \begin{subequations} 
\begin{align}
T''_{i,-\epsilon}\left(q^{h}\right) & =q^{s_{i}\left(h\right)},\\
T''_{i,-\epsilon}\left(e_{j}\right) & =\begin{cases}
-f_{i}k_{i}^{-\epsilon} & \text{for}\;j=i,\\
{\displaystyle \sum_{r+s=-\left\langle h_{i},\alpha_{j}\right\rangle }\left(-1\right)^{r}q_{i}^{\epsilon r}e_{i}^{\left(s\right)}e_{j}e_{i}^{\left(r\right)}} & \text{for}\;j\neq i,
\end{cases}\\
T''_{i,-\epsilon}\left(f_{j}\right) & =\begin{cases}
-k_{i}^{\epsilon}e_{i} & \text{for}\;j=i,\\
{\displaystyle \sum_{r+s=-\left\langle h_{i},\alpha_{j}\right\rangle }\left(-1\right)^{r}q_{i}^{-\epsilon r}f_{i}^{\left(r\right)}f_{j}f_{i}^{\left(s\right)}} & \text{for}\;j\neq i.
\end{cases}
\end{align}
\end{subequations}

It is known that $\left\{ T'_{i,\epsilon}\right\} _{i\in I}$ and
$\left\{ T''_{i,\epsilon}\right\} _{i\in I}$ satisfy the braid relation.
\begin{lem}[{\cite[Proposition 37.1.2 (d), Section 37.2.4]{Lus:intro}}]

\textup{(1)} We have $T'_{i,\epsilon}\circ T''_{i,-\epsilon}=T''_{i,-\epsilon}\circ T'_{i,\epsilon}=\mathrm{id}.$

\textup{(2)} We have $*\circ T'_{i,\epsilon}\circ*=T''_{i,-\epsilon}$
for $i\in I$ and $\epsilon\in\left\{ \pm1\right\} $.
\end{lem}
In the following, we write $T_{i}=T''_{i,1}$ and $T_{i}^{-1}=T'_{i,-1}$
as in \cite[Proposition 1.3.1]{Saito:PBW}.

\subsubsection{}

We have the following orthogonal decomposition with respect to the
bilinear form $\left(\cdot,\cdot\right)\colon\Uq^{-}\otimes\Uq^{-}\to\mathbb{Q}\left(q\right)$
and the compatibility with the canonical basis.
\begin{prop}[{\cite[Proposition 38.1.6, Lemma 38.1.5]{Lus:intro}}]
\label{prop:braidcrystal}

\textup{(1)} For $i\in I$, we have 
\begin{align*}
\Uq^{-}\cap T_{i}\Uq^{-} & =\left\{ x\in\Uq^{-}\mid{}_{i}r\left(x\right)=0\right\} ,\\
\Uq^{-}\cap T_{i}^{-1}\Uq^{-} & =\left\{ x\in\Uq^{-}\mid r_{i}\left(x\right)=0\right\} .
\end{align*}

\textup{(2)} For $i\in I$, we have the following orthogonal decomposition
with respect to $\left(\cdot,\cdot\right)_{-}$:
\[
\Uq^{-}=\left(\Uq^{-}\cap T_{i}\Uq^{-}\right)\oplus f_{i}\Uq^{-}=\left(\Uq^{-}\cap T_{i}^{-1}\Uq^{-}\right)\oplus\Uq^{-}f_{i}.
\]
\end{prop}
\begin{cor}
\label{cor:dualcan-simple}For $i\in I$, the subalgebra $\Uq^{-}\cap T_{i}\Uq^{-}$
(resp. $\Uq^{-}\cap T_{i}\Uq^{-}$) is compatible with the dual canonical
base, that is we have 
\begin{align*}
\Uq^{-}\cap T_{i}\Uq^{-}\cap\Uq^{-}\left(\mathfrak{g}\right)_{\mathcal{A}}^{\up} & =\bigoplus_{\substack{b\in\mathscr{B}\left(\infty\right)\\
\varepsilon_{i}\left(b\right)=0
}
}\mathcal{A}G^{\up}\left(b\right),\\
\Uq^{-}\cap T_{i}^{-1}\Uq^{-}\cap\Uq^{-}\left(\mathfrak{g}\right)_{\mathcal{A}}^{\up} & =\bigoplus_{\substack{b\in\mathscr{B}\left(\infty\right)\\
\varepsilon_{i}^{*}\left(b\right)=0
}
}\mathcal{A}G^{\up}\left(b\right).
\end{align*}

\end{cor}

\subsubsection{}

The following result is due to Saito \cite{Saito:PBW}.
\begin{prop}[{\cite[Proposition 3.4.7, Corollary 3.4.8]{Saito:PBW}}]

\textup{(1)} Let $x\in\Uq^{-}\in\mathscr{L}\left(\infty\right)\cap T_{i}^{-1}\Uq^{-}$
with $b:=x\;\mathrm{mod}\;q\mathscr{L}\left(\infty\right)\in\mathscr{B}\left(\infty\right)$,
we have 
\begin{align*}
T_{i}\left(x\right) & \in\mathscr{L}\left(\infty\right)\cap T_{i}\Uq^{-},\\
T_{i}\left(x\right) & \equiv\tilde{f}_{i}^{*\varphi_{i}\left(b\right)}\tilde{e}_{i}^{\varepsilon_{i}\left(b\right)}b\;\text{mod}\;q\mathscr{L}\left(\infty\right)\in\mathscr{B}\left(\infty\right).
\end{align*}

\textup{(2)} Let $\sigma_{i}\colon\left\{ b\in\mathscr{B}\left(\infty\right)\mid\varepsilon_{i}^{*}\left(b\right)=0\right\} \to\left\{ b\in\mathscr{B}\left(\infty\right)\mid\varepsilon_{i}\left(b\right)=0\right\} $
be the map defined by $\sigma_{i}\left(b\right)=\tilde{f}_{i}^{*\varphi_{i}\left(b\right)}\tilde{e}_{i}^{\varepsilon_{i}\left(b\right)}b$.
Then $\sigma_{i}$ is bijective and its inverse is given by $\sigma_{i}^{*}\left(b\right)=\left(*\circ\sigma_{i}\circ*\right)\left(b\right)=\tilde{f}_{i}^{\varphi_{i}^{*}\left(b\right)}\tilde{e}_{i}^{*\varepsilon_{i}^{*}\left(b\right)}b$. 
\end{prop}
The bijections $\sigma_{i}$ and $\sigma_{i}^{*}$ is called Saito
crystal reflections. In \cite[Corollary 3.4.8]{Saito:PBW}, $\sigma_{i}$
and $\sigma_{i}^{*}$ are denoted by $\Lambda_{i}$ and $\Lambda_{i}^{-1}$.

Following Baumann-Kamnitzer-Tingley \cite[Section 5.5]{BKT:affine},
for convenience, we extend $\sigma_{i}$ and $\sigma_{i}^{*}$ to
$\mathscr{B}\left(\infty\right)$ by setting 
\begin{align*}
\hat{\sigma}_{i}\left(b\right) & :=\sigma_{i}\left(\tilde{e}_{i}^{*\max}\left(b\right)\right),\\
\hat{\sigma}_{i}^{*}\left(b\right) & :=\sigma_{i}^{*}\left(\tilde{e}_{i}^{\max}\left(b\right)\right),
\end{align*}
so we can consider as maps from $\mathscr{B}\left(\infty\right)$
to itself.

\subsubsection{}

Let $^{i}\pi\colon\Uq^{-}\to\Uq^{-}\cap T_{i}\Uq^{-}$ (resp. $\pi^{i}\colon\Uq^{-}\to\Uq^{-}\cap T_{i}^{-1}\Uq^{-}$
) be the orthogonal projection whose kernel is $f_{i}\Uq^{-}$ (resp.
$\Uq^{-}f_{i}$). We have the following relations among the braid
group action and the (dual) canonical basis.
\begin{thm}[{\cite[Theorem 1.2]{Lus:braid}, \cite[Theorem 4.23]{Kim:qunip}}]
 \label{thm:braidcanbasis}

\textup{(1)} For $b\in\mathscr{B}\left(\infty\right)$ with $\varepsilon_{i}^{*}\left(b\right)=0$,
we have 
\begin{align*}
T_{i}\left(\pi^{i}G^{\low}\left(b\right)\right) & ={}^{i}\pi\left(G^{\low}\left(\sigma_{i}\left(b\right)\right)\right),\\
\left(1-q_{i}^{2}\right)^{\left\langle h_{i},\wt b\right\rangle }T_{i}G^{\up}\left(b\right) & =G^{\up}\left(\sigma_{i}b\right).
\end{align*}

\textup{(2)} For $b\in\mathscr{B}\left(\infty\right)$ with $\varepsilon_{i}\left(b\right)=0$,
we have 
\begin{align*}
T_{i}^{-1}\left(^{i}\pi G^{\low}\left(b\right)\right) & =\pi^{i}\left(G^{\low}\left(\sigma_{i}^{*}\left(b\right)\right)\right),\\
\left(1-q_{i}^{2}\right)^{\left\langle h_{i},\wt b\right\rangle }T_{i}^{-1}G^{\up}\left(b\right) & =G^{\up}\left(\sigma_{i}^{*}b\right).
\end{align*}

\end{thm}
We note that the constant term $\left(1-q_{i}^{2}\right)^{\left\langle h_{i},\wt b\right\rangle }$
depends on a choice of the non-degenerate bilinear form on $\mathbf{U}_{q}^{-}\left(\mathfrak{g}\right)$.

\subsection{\PBW bases}

Let $W=\left\langle s_{i}\mid i\in I\right\rangle $ be the Weyl group
of $\mathfrak{g}$ where $\left\{ s_{i}\mid i\in I\right\} $ is the
set of simple reflections associated with $i\in I$ and $\ell\colon W\to\mathbb{Z}_{\geq0}$
be the length function. For a Weyl group element $w$, let 
\[
I\left(w\right):=\left\{ \left(i_{1},i_{2},\cdots,i_{\ell\left(w\right)}\right)\in I^{\ell\left(w\right)}\mid s_{i_{1}}\cdots s_{i_{\ell\left(w\right)}}=w\right\} 
\]
be the set of reduced words of $w$.

\subsubsection{}

Let $\Delta=\Delta_{+}\sqcup\Delta_{-}$ be the root system of the
Kac-Moody Lie algebra $\mathfrak{g}$ and decomposition into positive
and negative roots.

For a Weyl group element $w\in W$, we set 
\begin{align*}
\Delta_{+}\left(\leq w\right) & :=\Delta_{+}\cap w\Delta_{-}=\left\{ \beta\in\Delta_{+}\mid w^{-1}\beta\in\Delta_{-}\right\} ,\\
\Delta_{+}\left(>w\right) & :=\Delta_{+}\cap w\Delta_{+}=\left\{ \beta\in\Delta_{+}\mid w^{-1}\beta\in\Delta_{+}\right\} .
\end{align*}

It is well-known that $\Delta_{+}\left(\leq w\right)$ and $\Delta_{+}\left(>w\right)$
are bracket closed, that is, for $\alpha,\beta\in\Delta_{+}\left(\leq w\right)$
(resp. $\alpha,\beta\in\Delta_{+}\left(>w\right)$) with $\alpha+\beta\in\Delta_{+}$,
we have $\alpha+\beta\in\Delta_{+}\left(w\right)$ (resp. $\in\Delta_{+}\left(>w\right)$).

For a reduced word $\bm{i}=\left(i_{1},i_{2},\cdots,i_{\ell}\right)\in I\left(w\right)$,
we define positive roots $\beta_{\bm{i},k}\;\left(1\leq k\leq\ell\right)$
by the following formula: 
\[
\beta_{\bm{i},k}=s_{i_{1}}\cdots s_{i_{k-1}}\left(\alpha_{i_{k}}\right)\;\left(1\leq k\leq\ell\right).
\]

It is well known that $\Delta_{+}\left(\leq w\right)=\left\{ \beta_{\bm{i},k}\right\} _{1\leq k\leq\ell}$
and we put a total order on $\Delta_{+}\left(w\right)$. We note that
the total order on $\Delta_{+}\left(\leq w\right)$ does depends on
a choice of a reduced word $\bm{i}\in I\left(w\right)$.

\subsubsection{}

For a Weyl group element $w\in W$ , a reduced word $\bm{i}=\left(i_{1},i_{2},\cdots,i_{\ell}\right)\in I\left(w\right)$
, we define the root vector $f_{\epsilon}\left(\beta_{\bm{i},k}\right)$
associated with $\beta_{\bm{i},k}\in\Delta_{+}\left(w\right)$ and
a sign $\epsilon\in\left\{ \pm1\right\} $ by 
\[
f_{\epsilon}\left(\beta_{\bm{i},k}\right):=T_{i_{1}}^{\epsilon}T_{i_{2}}^{\epsilon}\cdots T_{i_{k-1}}^{\epsilon}\left(f_{i_{k}}\right),
\]
 and its divided power 
\[
f_{\epsilon}\left(\beta_{\bm{i},k}\right)^{\left(c\right)}:=T_{i_{1}}^{\epsilon}T_{i_{2}}^{\epsilon}\cdots T_{i_{k-1}}^{\epsilon}\left(f_{i_{k}}^{\left(c\right)}\right)
\]
 for $c\in\mathbb{Z}_{\geq0}$.
\begin{thm}[{\cite[Proposition 40.2.1, 41.1.3]{Lus:intro}}]

For $w\in W$ , $\bm{w}=\left(i_{1},\cdots,i_{\ell}\right)\in I\left(w\right)$,
$\epsilon\in\left\{ \pm1\right\} $ and $\bm{c}\in\mathbb{Z}_{\geq0}^{\ell}$,
we set
\[
f_{\epsilon}\left(\bm{c},\bm{i}\right):=\begin{cases}
f_{\epsilon}\left(\beta_{\bm{i},1}\right)^{\left(c_{1}\right)}f_{\epsilon}\left(\beta_{\bm{i},2}\right)^{\left(c_{2}\right)}\cdots f_{\epsilon}\left(\beta_{\bm{i},\ell}\right)^{\left(c_{\ell}\right)} & \text{if}\;\epsilon=+1,\\
f_{\epsilon}\left(\beta_{\bm{i},\ell}\right)^{\left(c_{\ell}\right)}f_{\epsilon}\left(\beta_{\bm{i},\ell-1}\right)^{\left(c_{\ell-1}\right)}\cdots f_{\epsilon}\left(\beta_{\bm{i},1}\right)^{\left(c_{1}\right)} & \text{if}\;\epsilon=-1.
\end{cases}
\]
Then $\left\{ f_{\epsilon}\left(\bm{c},\bm{i}\right)\right\} _{\bm{c}\in\mathbb{Z}_{\geq0}^{\ell}}$
forms a basis of a subspace defined to be $\Uq^{-}\left(\leq w,\epsilon\right)$
of $\Uq^{-}\left(\mathfrak{g}\right)$ which does not depend on a
choice of a reduced word $\bm{i}\in I\left(w\right)$. $\left\{ f_{\epsilon}\left(\bm{c},\bm{i}\right)\right\} _{\bm{c}\in\mathbb{Z}_{\geq0}^{\ell}}$
is called the \PBW basis or the lower \PBW basis.\end{thm}
\begin{defn}
For a Weyl group element $w\in W$ , a reduced word $i=\left(i_{1},\cdots,i_{\ell}\right)\in I\left(w\right)$
and $c=\left(c_{1},\cdots,c_{\ell}\right)\in\mathbb{Z}_{\geq0}^{\ell}$,
we set 
\[
\xi\left(\bm{c},\bm{i}\right):=-\sum_{1\leq k\leq\ell}c_{k}\beta_{k,\bm{i}}\in Q_{-}.
\]

\end{defn}
We also have the following characterization of $\Uq^{-}\left(\leq w,\epsilon\right)$.
\begin{thm}[{\cite[Proposition 2.3]{BCP:affine}}]

For $w\in W$ and $\epsilon\in\left\{ \pm1\right\} $, we have 
\[
\Uq^{-}\left(\leq w,\epsilon\right)=\Uq^{-}\cap T_{w^{\epsilon}}^{\epsilon}\Uq^{\geq0}.
\]

In particular, $\Uq^{-}\left(\leq w,\epsilon\right)$ is a $\mathbb{Q}\left(q\right)$-subalgebra
of $\Uq^{-}$.
\end{thm}
In fact, it can be shown that $\Uq^{-}\left(\leq w,\epsilon\right)$
is a $\mathbb{Q}\left(q\right)$-subalgebra of $\Uq^{-}$ which is
generated by $\left\{ f_{\epsilon}\left(\beta_{\bm{i},k}\right)\right\} _{1\leq k\leq l}$
can be shown by the Levendorskii-Soibelman formula. For more details,
see \cite[Section 4.3]{Kim:qunip}.

\subsubsection{\PBW basis and crystal basis}
\begin{thm}
\label{thm:PBW}

For $w\in W$, $\bm{i}\in\left(i_{1},\cdots,i_{\ell}\right)\in I\left(w\right)$
and $\epsilon\in\left\{ \pm1\right\} $,

\textup{(1)} we have $f_{\epsilon}\left(\bm{c},\bm{i}\right)\in\mathscr{L}\left(\infty\right)$
and 
\[
b_{\epsilon}\left(\bm{c},\bm{i}\right):=f_{\epsilon}\left(\bm{c},\bm{i}\right)\;\text{mod}\;q\mathscr{L}\left(\infty\right)\in\mathscr{B}\left(\infty\right).
\]

\textup{(2)} The map $\mathbb{Z}_{\geq0}^{\ell}\to\mathscr{B}\left(\infty\right)$
which is defined by $\bm{c}\mapsto b_{\epsilon}\left(\bm{c},\bm{i}\right)$
is injective. We denote the image by $\mathscr{B}\left(w,\epsilon\right)$,
and this does not depend on a choice of a reduced word $\bm{i}\in I\left(w\right)$.
\end{thm}

\subsubsection{}
\begin{prop}[{\cite[Proposition 4.26 (2)]{Kim:qunip}}]
\label{prop:dualroot}

For $c\geq1$ and $1\leq k\leq\ell$, let 
\[
f_{\epsilon}^{\up}\left(\beta_{\bm{i},k}\right)^{\left\{ c\right\} }=f_{\epsilon}\left(\beta_{\bm{i},k}\right)^{\left(c\right)}/\left(f_{\epsilon}\left(\beta_{\bm{i},k}\right)^{\left(c\right)},f_{\epsilon}\left(\beta_{\bm{i},k}\right)^{\left(c\right)}\right),
\]
then we have $f_{\epsilon}^{\up}\left(\beta_{\bm{i},k}\right)^{\left\{ c\right\} }=q_{i_{k}}^{c\left(c-1\right)/2}f_{\epsilon}^{\up}\left(\beta_{\bm{i},k}\right)^{c}\in\mathbf{B}^{\up}$.\end{prop}
\begin{defn}[dual Poincare-Birkhoff-Witt basis]
 For $w\in W$ , $\bm{i}\in I\left(w\right)$ and $c\in\mathbb{Z}_{\geq0}^{\ell}$,
we set 
\[
f_{\epsilon}^{\up}\left(\bm{c},\bm{i}\right):=\dfrac{f_{\epsilon}\left(\bm{c},\bm{i}\right)}{\left(f_{\epsilon}\left(\bm{c},\bm{i}\right),f_{\epsilon}\left(\bm{c},\bm{i}\right)\right)}
\]
and $\left\{ f_{\epsilon}^{\up}\left(\bm{c},\bm{i}\right)\right\} _{\bm{c}\in\mathbb{Z}_{\geq0}^{\ell}}$
is called the dual \PBW basis or upper \PBW basis.

By the definition of the dual Poincare-Birkhoff-Witt basis and the
computation of $\left(f_{\epsilon}\left(\bm{c},\bm{i}\right),f_{\epsilon}\left(\bm{c},\bm{i}\right)\right)$,
we have
\begin{align*}
f_{\epsilon}^{\up}\left(\bm{c},\bm{i}\right) & =\begin{cases}
f_{\epsilon}^{\up}\left(\beta_{\bm{i},1}\right)^{\left[c_{1}\right]}f_{\epsilon}^{\up}\left(\beta_{\bm{i},2}\right)^{\left[c_{2}\right]}\cdots f_{\epsilon}^{\up}\left(\beta_{\bm{i},\ell}\right)^{\left[c_{\ell}\right]} & \text{if}\;\epsilon=+1,\\
f_{\epsilon}^{\up}\left(\beta_{\bm{i},\ell}\right)^{\left[c_{\ell}\right]}f_{\epsilon}^{\up}\left(\beta_{\bm{i},\ell-1}\right)^{\left[c_{\ell-1}\right]}\cdots f_{\epsilon}^{\up}\left(\beta_{\bm{i},1}\right)^{\left[c_{1}\right]} & \text{if}\;\epsilon=-1.
\end{cases}\\
 & =\begin{cases}
\left(1-q_{i_{1}}^{2}\right)^{\left\langle h_{i_{1}},\xi\left(c_{\geq2},i_{\geq2}\right)\right\rangle }f_{\epsilon}^{\up}\left(\beta_{\bm{i},1}\right)^{\left[c_{1}\right]}T_{i_{1}}^{\epsilon}\left(f_{\epsilon}^{\up}\left(\beta_{\bm{i}_{\geq2},2}\right)^{\left[c_{2}\right]}\cdots f_{\epsilon}^{\up}\left(\beta_{\bm{i}_{\geq2},\ell}\right)^{\left[c_{\ell}\right]}\right) & \text{if}\;\epsilon=+1,\\
\left(1-q_{i_{1}}^{2}\right)^{\left\langle h_{i_{1}},\xi\left(c_{\geq2},i_{\geq2}\right)\right\rangle }T_{i_{1}}^{\epsilon}\left(f_{\epsilon}^{\up}\left(\beta_{\bm{i}_{\geq2},2}\right)^{\left[c_{2}\right]}\cdots f_{\epsilon}^{\up}\left(\beta_{\bm{i}_{\geq2},\ell}\right)^{\left[c_{\ell}\right]}\right)f_{\epsilon}^{\up}\left(\beta_{\bm{i},1}\right)^{\left[c_{1}\right]} & \text{if}\;\epsilon=-1.
\end{cases}
\end{align*}
where $\bm{c}_{\geq2}=\left(c_{2},\cdots,c_{\ell}\right)\in\mathbb{Z}_{\geq0}^{\ell-1}$
, $w_{\geq2}=s_{i_{2}}\cdots s_{i_{\ell}}$ and $i_{\geq2}=\left(i_{2},\cdots,i_{\ell}\right)\in I\left(w_{\geq2}\right)$.
\end{defn}
Using the Levendorskii-Soibelman formula (see \cite[Theorem 4.27]{Kim:qunip})
and the definition of the dual canonical basis, we have the following
result.
\begin{thm}[{\cite[Theorem 4.25, Theorem 4.29]{Kim:qunip}}]
\label{thm:qunip}

Let $w\in W$ and $\bm{i}\in I\left(w\right)$, the \PBW basis satisfying
the following properties

\textup{(1)} The subalgebra $\Uq^{-}\left(\leq w,\epsilon\right)$
is compatible with the dual canonical basis, that is there exists
a subset $\mathscr{B}\left(\leq w,\epsilon\right):=\mathscr{B}\left(\Uq^{-}\left(\leq w,\epsilon\right)\right)\subset\mathscr{B}\left(\infty\right)$
such that 
\[
\Uq^{-}\left(\leq w,\epsilon\right)=\bigoplus_{b\in\mathscr{B}\left(\leq w,\epsilon\right)}\mathbb{Q}\left(q\right)G^{\up}\left(b\right).
\]

\textup{(2)} The transition matrix between the dual \PBW basis and
the dual canonical basis is triangular with 1's on the diagonal with
respect to the (left) lexicographic order $\leq$ on $\mathbb{Z}_{\geq0}^{\ell}$.
More precisely, we have 
\[
f_{\epsilon}^{\up}\left(\bm{c},\bm{i}\right)=G^{\up}\left(b_{\epsilon}\left(\bm{c},\bm{i}\right)\right)+\sum_{\bm{c}^{'}<\bm{c}}d_{\bm{c},\bm{c}^{'}}^{\bm{i}}\left(q\right)G^{\up}\left(b_{\epsilon}\left(\bm{c}^{'},\bm{i}\right)\right)
\]

with $d_{\bm{c},\bm{c}^{'}}^{\bm{i}}\left(q\right):=\left(f_{\epsilon}^{\up}\left(\bm{c},\bm{i}\right),G^{\low}\left(b_{\epsilon}\left(\bm{c}^{'},\bm{i}\right)\right)\right)\in q\mathbb{Z}\left[q\right]$.\end{thm}
\begin{rem}
In symmetric case, we note that it can be shown that 
\[
d_{\bm{c},\bm{c}^{'}}^{\bm{i}}\left(q\right)=\left(f_{\epsilon}^{\up}\left(\bm{c},\bm{i}\right),G^{\low}\left(b_{\epsilon}\left(\bm{c}^{'},\bm{i}\right)\right)\right)\in q\mathbb{Z}_{\geq0}\left[q\right],
\]
by the positivity of the (twisted) comultiplication with respect to
the canonical basis and Proposition \ref{prop:dualroot}.

In particular we obtain a proof of the positivity of the transition
matrix from the canonical basis into the lower \PBW basis in simply-laced
type for arbitrary reduced word of the longest element $w_{0}$ using
the orthogonality of the the (lower) \PBW basis.

For ``adapted'' reduced words, it was proved by Lusztig \cite[Corollary 10.7]{MR1035415}.
For an arbitrary reduced word, it is proved by Kato \cite[Theorm 4.17]{MR3165425}
using the categorification of \PBW basis via Khovanov-Lauda-Rouquier
algebra. It is also proved by Oya \cite[Theorem 5.2]{Oya:qfa}.
\end{rem}
\begin{NB}

\subsection{Integrality of the inner product}

Since $\mathbf{B}^{\mathrm{up}}\left(\leq w,\epsilon\right):=\mathbf{B}^{\mathrm{up}}\cap\Uq^{-}\left(\leq w,\epsilon\right)$
is a $\mathbb{Q}\left(q\right)$-basis of $\Uq^{-}\left(\leq w,\epsilon\right)$,
we call it the dual canonical basis of $\Uq^{-}\left(\leq w,\epsilon\right)$.
\begin{cor}
We have $\left(\mathbf{B}^{\mathrm{up}}\left(\leq w,\epsilon\right),\mathbf{B}^{\mathrm{up}}\left(\leq w,\epsilon\right)\right)_{L}\subset\mathbb{Z}\left[q^{\pm1}\right]$.\end{cor}
\begin{proof}
We have $\left(F^{\up}\left(\bm{w}\right),F^{\up}\left(\bm{w}\right)\right)_{L}\subset\mathbb{Z}\left[q^{\pm1}\right]$
by the computation, hence the integral recursion yields the result.
\end{proof}
\end{NB}

\begin{NB}

\subsection{Conjecture on quantum unipotent subgroup and its localization}

In \cite[5.1.3]{Kim:qunip}, we introduced quantum closed unipotent
cell $\mathcal{O}_{q}\left[\overline{N_{w}}\right]$ and proved the
existence of injective ring homomorphism $\Uq^{-}\left(\leq w,-1\right)_{\mathcal{A}}^{\up}\hookrightarrow\mathcal{O}_{q}\left[\overline{N_{w}}\right]$
in \cite[Theorem 5.13]{Kim:qunip}. We note that $\mathcal{O}_{q}\left[\overline{N_{w}}\right]$
is also tread in Berenstein-Rupel \cite{BR:Hall} as a image of the
Feigin map. This is a generalization of De Concini and Procesi's result
on finite type case.
\begin{conjecture}
\textup{(1)} After the localizations with respect to $\left\{ \Delta_{w\lambda,\lambda}\mid\lambda\in P_{+}\right\} $,
we have the induced isomorphism:
\[
\Uq^{-}\left(\leq w,-1\right)_{\mathcal{A}}^{\up}\left[\Delta_{w\lambda,\lambda}^{-1}\mid\lambda\in P_{+}\right]\xrightarrow{\sim}\mathcal{O}_{q}\left[\overline{N_{w}}\right]\left[\Delta_{w\lambda,\lambda}^{-1}\mid\lambda\in P_{+}\right].
\]

\textup{(2)} After the localization with with respect to $\left\{ \Delta_{w\lambda,\lambda}\mid\lambda\in P_{+}\right\} $,
we have the identification of the dual canonical bases in the both
hand side.
\end{conjecture}
\end{NB}

\section{Proof of the surjectivity}

\subsection{Multiplication formula for $\Uq^{-}\left(\leq w,\epsilon\right)$.}

For a Weyl group element $w$, a reduced word $\bm{i}\in I\left(w\right)$
and $0\leq p<\ell$, we consider a subalgebra which is generated by
$\left\{ f_{\epsilon}^{\up}\left(\beta_{\bm{i},k}\right)\right\} _{p+1\leq k\leq\ell}$,
then it can be shown that this subalgebra is also compatible with
the dual canonical basis. This can be proved using the transition
matrix between the dual \PBW basis and the dual canonical basis.

In this subsection, we give statements for the $\epsilon=+1$ case.
We can obtain the corresponding claims for $\epsilon=-1$ case by
applying the $*$-involution. So we denote $f_{\epsilon}^{\up}\left(\beta_{\bm{i},k}\right)$,
$f_{\epsilon}^{\up}\left(\bm{c},\bm{i}\right)$ , $b_{\epsilon}\left(\bm{c},\bm{i}\right)$
by $f^{\up}\left(\beta_{\bm{i},k}\right)$, $f^{\up}\left(\bm{c},\bm{i}\right)$,
$b\left(\bm{c},\bm{i}\right)$ by omitting $\epsilon$.
\begin{prop}
Let $w\in W$ and $\bm{i}\in I\left(w\right)$. For $\bm{c}\in\mathbb{Z}_{\geq0}^{\ell}$
and $0\leq p<\ell$, we set 
\begin{align*}
\tau_{\leq p}\left(\bm{c}\right) & :=\left(c_{1},\cdots,c_{p},0,\cdots,0\right)\in\mathbb{Z}_{\geq0}^{\ell},\\
\tau_{>p}\left(\bm{c}\right) & :=\left(0,\cdots,0,c_{p+1},\cdots,c_{\ell}\right)\in\mathbb{Z}_{\geq0}^{\ell},
\end{align*}
then we have 
\[
G^{\up}\left(b\left(\tau_{\leq p}\left(\bm{c}\right)\right),\bm{i}\right)G^{\up}\left(b\left(\tau_{>p}\left(\bm{c}\right),\bm{i}\right)\right)\in G^{\up}\left(b\left(\bm{c},\bm{i}\right)\right)+\sum_{\bm{d}<\bm{c}}q\mathbb{Z}\left[q\right]G^{\up}\left(b\left(\bm{d},\bm{i}\right)\right).
\]
\end{prop}
\begin{proof}
By the transition from the dual canonical basis to the dual \PBW
basis, we have 
\begin{align*}
G^{\up}\left(b\left(\tau_{\leq p}\left(\bm{c}\right)\right),\bm{i}\right) & \in f^{\up}\left(\tau_{\leq p}\left(\bm{c}\right),\bm{i}\right)+\sum_{\bm{d}_{\leq p}\leq\tau_{\leq p}\left(\bm{c}\right)}q\mathbb{Z}\left[q\right]f^{\up}\left(\bm{d}_{\leq p},\bm{i}\right),\\
G^{\up}\left(b\left(\tau_{>p}\left(\bm{c}\right),\bm{i}\right)\right) & \in f^{\up}\left(\tau_{>p}\left(\bm{c}\right),\bm{i}\right)+\sum_{\bm{d}_{>p}\leq\tau_{>p}\left(\bm{c}\right)}q\mathbb{Z}\left[q\right]f^{\up}\left(\bm{d}_{>p},\bm{i}\right).
\end{align*}
and we note that we have $\bm{d}_{\leq p}=\tau_{\leq p}\left(\bm{d}_{\leq p}\right)$
and $\bm{d}_{>p}=\tau_{>p}\left(\bm{d}_{>p}\right)$ by the Levendorskii-Soibelman
formula in the right hand sides.

Hence in the product of the right hand side, we have the following
4 kinds of terms:
\begin{align*}
f^{\up}\left(\bm{c},\bm{i}\right) & =f^{\up}\left(\tau_{\leq p}\left(\bm{c}\right),\bm{i}\right)f^{\up}\left(\tau_{>p}\left(\bm{c}\right),\bm{i}\right),\\
f^{\up}\left(\tau_{\leq p}\left(\bm{c}\right)+\bm{d}_{>p},\bm{i}\right) & =f^{\up}\left(\tau_{\leq p}\left(\bm{c}\right),\bm{i}\right)f^{\up}\left(\bm{d}_{>p},\bm{i}\right),\\
f^{\up}\left(\tau_{>p}\left(\bm{c}\right)+\bm{d}_{\leq p},\bm{i}\right) & =f^{\up}\left(\bm{d}_{\leq p},\bm{i}\right)f^{\up}\left(\tau_{>p}\left(\bm{c}\right),\bm{i}\right),\\
f^{\up}\left(\bm{d}_{\leq p}+\bm{d}_{>p},\bm{i}\right) & =f^{\up}\left(\bm{d}_{\leq p},\bm{i}\right)f^{\up}\left(\bm{d}_{>p},\bm{i}\right).
\end{align*}

We note that $\tau_{\leq p}\left(\bm{c}\right)+\bm{d}_{>p}<_{\bm{w}}\bm{c}$,
$\tau_{>p}\left(\bm{c}\right)+\bm{d}_{\leq p}<_{\bm{w}}\bm{c}$, $\bm{d}_{\leq p}+\bm{d}_{>p}<_{\bm{w}}\bm{c}$
by construction, so we have $\tau_{\leq p}\left(\bm{c}\right)+\bm{d}_{>p}<\bm{c}$,
$\tau_{>p}\left(\bm{c}\right)+\bm{d}_{\leq p}<\bm{c}$ and $\bm{d}_{\leq p}+\bm{d}_{>p}<_{\bm{w},p}\bm{c}$.
Hence, using the transition from the dual \PBW basis to the dual
canonical basis, we obtain the claim.
\end{proof}

\subsection{Compatibility of $\Uq^{-}\left(>w,\epsilon\right)$}

For a Weyl group element, we consider the co-finite subset $\Delta_{+}\cap w\Delta_{+}$
and corresponding quantum coordinate ring $\Uq^{-}\left(>w,\epsilon\right)$.
\begin{defn}
For $w\in W$ and $\epsilon\in\left\{ \pm1\right\} $, we set 
\[
\Uq^{-}\left(>w,\epsilon\right)=\Uq^{-}\cap T_{w^{\epsilon}}^{\epsilon}\Uq^{-}.
\]

\end{defn}
Using Proposition \ref{prop:braidcrystal} iteratively, we obtain
the following lemma.

The following is the main result in this subsection.
\begin{thm}
For $w\in W$ and $\epsilon\in\left\{ \pm1\right\} $, $\Uq^{-}\left(>w,\epsilon\right)$
is compatible with the dual canonical basis, $\mathbf{B}^{\mathrm{up}}\left(>w,\epsilon\right):=\mathbf{B}^{\mathrm{up}}\cap\Uq^{-}\left(>w,\epsilon\right)$
is a $\mathbb{Q}\left(q\right)$-basis of $\Uq^{-}\left(>w,\epsilon\right)$.
\end{thm}
The proof of this theorems occupies the rest of this subsection and
we give the characterization of the subset $\mathbf{B}^{\mathrm{up}}\left(>w,\epsilon\right)$.

\subsubsection{}

First, we give a variant of definition of $\Uq^{-}\left(>w,\epsilon\right)$
which depends on $\bm{i}=\left(i_{1},\cdots,i_{\ell}\right)\in I\left(w\right)$
which is suitable for the description of the dual canonical basis.
\begin{prop}
\label{prop:cofinsubgroup}For $w\in W$, $\bm{i}=\left(i_{1},\cdots,i_{\ell}\right)\in I\left(w\right)$
and $\epsilon\in\left\{ \pm1\right\} $, we have 
\[
\Uq^{-}\left(>w,\epsilon\right)=\Uq^{-}\cap T_{i_{1}}^{\epsilon}\Uq^{-}\cap T_{i_{1}}^{\epsilon}T_{i_{2}}^{\epsilon}\Uq^{-}\cap\cdots\cap T_{i_{1}}^{\epsilon}\cdots T_{i_{\ell}}^{\epsilon}\Uq^{-}.
\]

\end{prop}
In fact, the right hand side does not depend on a choice of a reduced
word $\bm{i}\in I\left(w\right)$. The above proposition can be shown
clearly by the following Lemmas.
\begin{lem}
\label{lem:PBWdecomp}For a Weyl group element $w\in W$ and a reduced
word $\bm{i}=\left(i_{1},\cdots,i_{\ell}\right)\in I\left(w\right)$
and homogenous element $x\in\Uq^{-}$, there exists $x_{\bm{c}}\in\Uq^{-}\cap T_{i_{\ell}}\Uq^{-}$
for $\bm{c}\in\mathbb{Z}_{\geq0}^{\ell}$ with 
\begin{align}
T_{w}^{-1}\left(x\right) & =\sum_{\bm{c}\in\mathbb{Z}_{\geq0}^{\ell}}T_{i_{\ell}}^{-1}\cdots T_{i_{1}}^{-1}\left(f_{i_{1}}^{\left(c_{1}\right)}\right)\cdots T_{i_{\ell}}^{-1}T_{i_{\ell-1}}^{-1}\left(f_{i_{\ell-1}}^{\left(c_{\ell-1}\right)}\right)T_{i_{\ell}}^{-1}\left(f_{i_{\ell}}^{\left(c_{\ell}\right)}\right)T_{i_{\ell}}^{-1}\left(x_{\bm{c}}\right)\label{eq:PBWdecomp}\\
 & \in\sum_{\bm{c}\in\mathbb{Z}_{\geq0}^{\ell}}T_{i_{\ell}}^{-1}\cdots T_{i_{2}}^{-1}\left(e_{i_{1}}^{\left(c_{1}\right)}\right)\cdots T_{i_{\ell}}^{-1}\left(e_{i_{\ell-1}}^{\left(c_{\ell-1}\right)}\right)e_{i_{\ell}}^{\left(c_{\ell}\right)}\Uq^{\leq0}.\nonumber 
\end{align}
\end{lem}
\begin{rem}
We note that it is not clear that $T_{i_{\ell}}^{-1}\left(x_{\bm{c}}\right)\in\Uq^{-}\cap T_{w}^{-1}\Uq^{-}$
in the right hand side of (\ref{eq:PBWdecomp}). But, in fact, it
can be proved by the surjectivity.
\end{rem}
\begin{NB}
\begin{proof}
This is an argument in $\ell\left(w\right)=1$ case, any element $x\in\Uq^{-}$
can be written 
\[
x=\sum f_{i}^{\left(c\right)}x_{c}\;\text{with}\;x_{c}\in\Uq^{-}\cap T_{i}\Uq^{-}.
\]
We consider the decomposition:
\[
x=\sum f_{i_{1}}^{\left(c_{1}\right)}x_{c_{1}}\;\text{with}\;x_{c_{1}}\in\Uq^{-}\cap T_{i_{1}}\Uq^{-}.
\]

So we have $T_{i_{1}}^{-1}\left(x_{c_{1}}\right)\in\mathbf{U}_{q}^{-}\cap T_{i_{1}}^{-1}\mathbf{U}_{q}^{-1}$.
Applying $T_{i_{1}}^{-1}$, we obtain 
\[
T_{i_{1}}^{-1}x=\sum_{c_{1}\geq0}T_{i_{1}}^{-1}\left(f_{i_{1}}^{\left(c_{1}\right)}\right)T_{i_{1}}^{-1}\left(x_{c_{1}}\right).
\]
with so we obtain the claim. For $T_{i_{1}}^{-1}\left(x_{c_{1}}\right)\in\Uq^{-}$,
we consider the decomposition 
\[
T_{i_{1}}^{-1}\left(x_{c_{1}}\right)=\sum_{c_{2}\geq0}f_{i_{2}}^{\left(c_{2}\right)}x_{c_{1},c_{2}}\;\text{with}\;x_{c_{1},c_{2}}\in\Uq^{-}\cap T_{i_{2}}\Uq^{-}
\]
 so we obtain 
\[
T_{i_{2}}^{-1}T_{i_{1}}^{-1}\left(x_{c_{1}}\right)=\sum_{c_{2}\geq0}T_{i_{2}}^{-1}\left(f_{i_{2}}^{\left(c_{2}\right)}\right)T_{i_{2}}^{-1}\left(x_{c_{1},c_{2}}\right)
\]
hence we have 
\begin{align*}
T_{i_{2}}^{-1}T_{i_{1}}^{-1}x & =T_{i_{2}}^{-1}\left(\sum_{c_{1}\geq0}T_{i_{1}}^{-1}\left(f_{i_{1}}^{\left(c\right)}\right)T_{i_{1}}^{-1}\left(x_{c_{1}}\right)\right)\\
 & =\sum_{\substack{c_{1}\geq0\\
c_{2}\geq0
}
}T_{i_{2}}^{-1}T_{i_{1}}^{-1}\left(f_{i_{1}}^{\left(c_{1}\right)}\right)T_{i_{2}}^{-1}\left(f_{i_{2}}^{\left(c_{2}\right)}\right)T_{i_{2}}^{-1}\left(x_{c_{1},c_{2}}\right).
\end{align*}
Iterating the construction along $\bm{w}$, we obtain 
\[
T_{i_{\ell}}^{-1}\cdots T_{i_{1}}^{-1}x=\sum_{\bm{c}\in\mathbb{Z}_{\geq0}^{\ell}}T_{i_{\ell}}^{-1}\cdots T_{i_{1}}^{-1}\left(f_{i_{1}}^{\left(c_{1}\right)}\right)\cdots T_{i_{\ell}}^{-1}T_{i_{\ell-1}}^{-1}\left(f_{i_{\ell-1}}^{\left(c_{\ell-1}\right)}\right)T_{i_{\ell}}^{-1}\left(f_{i_{\ell}}\right)T_{i_{\ell}}^{-1}\left(x_{c_{1},\cdots,c_{\ell}}\right)
\]

Then we obtain the claim.

\end{proof}
\end{NB}
\begin{lem}
If $\ell\left(s_{i}w\right)>\ell\left(w\right)$, we have $\Uq^{-}\cap T_{s_{i}w}\Uq^{-}\subset\Uq^{-}\cap T_{i}\Uq^{-}\cap T_{w}\Uq^{-}$.\end{lem}
\begin{proof}
For a homogenous element $x\in\Uq^{-}$, we decompose $x=\sum_{c\geq0}f_{i}^{\left(c\right)}x_{c}$
with $x_{c}\in\Uq^{-}\cap T_{i}\Uq^{-}$. So we have
\begin{align*}
T_{i}^{-1}x & =\sum_{c\geq0}T_{i}^{-1}\left(f_{i}^{\left(c\right)}\right)T_{i}^{-1}\left(x_{c}\right)\in\sum_{c\geq0}e_{i}^{\left(c\right)}\Uq^{\leq0}
\end{align*}
with $T_{i}^{-1}\left(x_{c}\right)\in\Uq^{-}\cap T_{i}^{-1}\Uq^{-}$.
Apply $T_{w}^{-1}$ in the both side, we have 
\begin{align*}
T_{w}^{-1}T_{i}^{-1}x & =\sum_{c\geq0}T_{w}^{-1}\left(f_{i}^{\left(c\right)}\right)T_{w}^{-1}\left(x_{c}\right)\\
 & \in\sum_{c\geq0}T_{i_{\ell}}^{-1}\cdots T_{i_{1}}^{-1}\left(e_{i}^{\left(c\right)}\right)T_{i_{\ell}}^{-1}\cdots T_{i_{2}}^{-1}\left(e_{i_{1}}^{\left(c_{1}\right)}\right)\cdots T_{i_{\ell}}^{-1}\left(e_{i_{\ell-1}}^{\left(c_{\ell-1}\right)}\right)e_{i_{\ell}}^{\left(c_{\ell}\right)}\Uq^{\leq0}
\end{align*}
For $x\in\Uq^{-}\cap T_{s_{i}w}\Uq^{-}$, we have $T_{w}^{-1}T_{i}^{-1}x\in\Uq^{-}\cap T_{s_{i}w}^{-1}\Uq^{-}$.
Since 
\[
\left\{ T_{i_{\ell}}^{-1}\cdots T_{i_{1}}^{-1}\left(e_{i}^{\left(c\right)}\right)T_{i_{\ell}}^{-1}\cdots T_{i_{2}}^{-1}\left(e_{i_{1}}^{\left(c_{1}\right)}\right)\cdots T_{i_{\ell}}^{-1}\left(e_{i_{\ell-1}}^{\left(c_{\ell-1}\right)}\right)e_{i_{\ell}}^{\left(c_{\ell}\right)}\middle|\left(c,c_{1},\cdots,c_{\ell}\right)\in\mathbb{Z}_{\geq0}^{\ell+1}\right\} 
\]
is linearly independent by the assumption $\ell\left(s_{i}w\right)>\ell\left(w\right)$,
hence we should have $x_{c}=0$ for $c>0$. In particular $x=x_{0}\in\Uq^{-}\cap T_{i}\Uq^{-}$
and also $T_{w}^{-1}\left(X_{0}\right)\in\Uq^{-}\cap T_{w}^{-1}\Uq^{-}$.
So $x=x_{0}\in\Uq^{-}\cap T_{i}\Uq^{-}\cap T_{w}\Uq^{-}$.
\end{proof}

\subsubsection{}

Let $w$ be a Weyl group element and $\bm{i}=\left(i_{1},\cdots,i_{\ell}\right)\in I\left(w\right)$
be a reduced word. Following Saito \cite[Lemma 4.1.3]{Saito:PBW}
and Baumann-Kamnitzer-Tingley \cite[Proposition 5.24]{BKT:affine},
we define Lusztig datum of $b\in\mathscr{B}\left(\infty\right)$ in
direction $\bm{i}\in I\left(w\right)$ and $\epsilon\in\left\{ \pm1\right\} $
($\left(\bm{i},\epsilon\right)$-Lusztig datum for short).
\begin{defn}[$\left(\bm{i},\epsilon\right)$-Lusztig datum]
 For $w\in W$, $\bm{i}\in I\left(w\right)$ and $\epsilon\in\left\{ \pm1\right\} $,
we define
\begin{align*}
L_{\epsilon}\left(b,\bm{i}\right) & =\begin{cases}
\left(\varepsilon_{i_{1}}\left(b\right),\varepsilon_{i_{2}}\left(\hat{\sigma}_{i_{1}}^{*}b\right),\cdots,\varepsilon_{i_{\ell}}\left(\hat{\sigma}_{i_{\ell-1}}^{*}\cdots\hat{\sigma}_{i_{1}}^{*}b\right)\right)\in\mathbb{Z}_{\geq0}^{\ell} & \epsilon=+1\\
\left(\varepsilon_{i_{1}}^{*}\left(b\right),\varepsilon_{i_{2}}^{*}\left(\hat{\sigma}_{i_{1}}b\right),\cdots,\varepsilon_{i_{\ell}}^{*}\left(\hat{\sigma}_{i_{\ell-1}}\cdots\hat{\sigma}_{i_{1}}b\right)\right)\in\mathbb{Z}_{\geq0}^{\ell} & \epsilon=-1
\end{cases}
\end{align*}

\end{defn}
By construction in \ref{thm:PBW}, we have 
\[
\bm{c}=L_{\epsilon}\left(b_{\epsilon}\left(\bm{c},\bm{i}\right),\bm{i}\right)
\]
for $\bm{c}\in\mathbb{Z}_{\geq0}^{\ell}$, that is the map $b_{\epsilon}\left(-,\bm{i}\right)\colon\mathbb{Z}_{\geq0}^{\ell}\to\mathscr{B}\left(\infty\right)$
is a section of $\left(\bm{i},\epsilon\right)$-Lusztig datum $L_{\epsilon}\left(-,\bm{i}\right)\colon\mathscr{B}\left(\infty\right)\to\mathbb{Z}_{\geq0}^{\ell}$.

\subsubsection{}

The following gives a characterization of $\mathbf{B}^{\mathrm{up}}\left(>w,\epsilon\right)$
in terms of the $\left(\bm{i},\epsilon\right)$-Lusztig data.
\begin{thm}
For $w\in W$ and $\bm{i}=\left(i_{1},\cdots,i_{\ell}\right)\in I\left(w\right)$,
we set 
\[
\mathscr{B}\left(>w,\epsilon\right)=\left\{ b\in\mathscr{B}\left(\infty\right)\mid L_{\epsilon}\left(b,\bm{i}\right)=0\right\} ,
\]
then we have 
\[
\Uq^{-}\left(>w,\epsilon\right)=\bigoplus_{b\in\mathscr{B}\left(>w,\epsilon\right)}\mathbb{Q}\left(q\right)G^{\up}\left(b\right).
\]
\end{thm}
\begin{proof}
By the Proposition \ref{prop:cofinsubgroup}, it suffices for us to
prove the compatibility for the intersection $\Uq^{-}\cap T_{i_{1}}^{\epsilon}\Uq^{-}\cap T_{i_{1}}^{\epsilon}T_{i_{2}}^{\epsilon}\Uq^{-}\cap\cdots\cap T_{i_{1}}^{\epsilon}\cdots T_{i_{\ell}}^{\epsilon}\Uq^{-}$.

Since $\epsilon=-1$ can be obtained by applying the $*$-involution,
we only prove $\epsilon=1$ case. We prove the claim by the induction
on the length $\ell\left(w\right)$. For $\ell\left(w\right)=1$,
it is the claim in Corollary \ref{cor:dualcan-simple}. We consider
the following intersection:
\[
\Uq^{-}\cap T_{i_{1}}^{-1}\Uq^{-}\cap T_{i_{2}}\Uq^{-}\cap\cdots\cap T_{i_{2}}\cdots T_{i_{\ell}}\Uq^{-};
\]
By the assumption of the induction on length, we know that $\Uq^{-}\cap T_{i_{2}}\Uq^{-}\cap\cdots\cap T_{i_{2}}\cdots T_{i_{\ell}}\Uq^{-}$
is compatible with the dual canonical basis and also $\Uq^{-}\cap T_{i_{1}}^{-1}\Uq^{-}$
is compatible with the dual canonical basis, hence the intersection
$\Uq^{-}\cap T_{i_{1}}^{-1}\Uq^{-}\cap T_{i_{2}}\Uq^{-}\cap\cdots\cap T_{i_{2}}\cdots T_{i_{\ell}}\Uq^{-}$
is compatible with the dual canonical basis. Applying Theorem \ref{thm:braidcanbasis},
we obtain the claim for $\Uq^{-}\cap T_{i_{1}}\Uq^{-}\cap T_{i_{1}}T_{i_{2}}\Uq^{-}\cap\cdots\cap T_{i_{1}}\cdots T_{i_{\ell}}\Uq^{-}$.
Since $\Uq^{-}\cap T_{i_{1}}\Uq^{-}\cap T_{i_{1}}T_{i_{2}}\Uq^{-}\cap\cdots\cap T_{i_{1}}\cdots T_{i_{\ell}}\Uq^{-}=\Uq^{-}\cap T_{i_{1}}\left(\Uq^{-}\cap T_{i_{2}}\Uq^{-}\cap\cdots\cap T_{i_{2}}\cdots T_{i_{\ell}}\Uq^{-}\right)$,
we obtain the description of $\mathbf{B}^{\mathrm{up}}\left(>w,+1\right)$.
\end{proof}

\subsection{Multiplication formula between $\mathbf{B}^{\mathrm{up}}\left(\leq w,\epsilon\right)$
and $\mathbf{B}^{\mathrm{up}}\left(>w,\epsilon\right)$}

\subsubsection{}

We generalize the (special cases of) formula in Theorem \ref{thm:multdualChevalley}
using the dual canonical basis $\mathbf{B}^{\mathrm{up}}\left(>w,\epsilon\right)$.
\begin{thm}
\label{thm:dualPBWmult}

For $b\in\mathscr{B}\left(>w,\epsilon\right)$ and $\bm{c}\in\mathbb{Z}_{\geq0}$,
we have
\begin{align*}
f_{\epsilon}^{\up}\left(\bm{c},\bm{i}\right)G^{\up}\left(b\right) & \in G^{\up}\left(\nabla_{\bm{i},\epsilon}^{\bm{c}}\left(b\right)\right)+\sum_{L_{\epsilon}\left(b^{'},\bm{i}\right)<\bm{c}}q\mathbb{Z}\left[q\right]G^{\up}\left(b^{'}\right)\;\text{if}\;\epsilon=+1,\\
G^{\up}\left(b\right)f_{\epsilon}^{\up}\left(\bm{c},\bm{i}\right) & \in G^{\up}\left(\nabla_{\bm{i},\epsilon}^{\bm{c}}\left(b\right)\right)+\sum_{L_{\epsilon}\left(b^{'},\bm{i}\right)<\bm{c}}q\mathbb{Z}\left[q\right]G^{\up}\left(b^{'}\right)\;\text{if}\;\epsilon=-1,
\end{align*}
where
\[
\nabla_{\bm{i},\epsilon}^{\bm{c}}\left(b\right)=\begin{cases}
\tilde{f}_{i_{1}}^{c_{1}}\sigma_{i_{1}}\cdots\tilde{f}_{i_{\ell-1}}^{c_{\ell-1}}\sigma_{i_{\ell-1}}\tilde{f}_{i_{\ell}}^{c_{\ell}}\sigma_{i_{\ell}}\sigma_{i_{\ell}}^{*}\cdots\sigma_{i_{1}}^{*}\left(b\right) & \text{if}\;\epsilon=+1,\\
\tilde{f}_{i_{1}}^{*c_{1}}\sigma_{i_{1}}^{*}\cdots\tilde{f}_{i_{\ell-1}}^{*c_{\ell-1}}\sigma_{i_{\ell-1}}^{*}\tilde{f}_{i_{\ell}}^{*c_{\ell}}\sigma_{i_{\ell}}^{*}\sigma_{i_{\ell}}\cdots\sigma_{i_{1}}\left(b\right) & \text{if}\;\epsilon=-1.
\end{cases}
\]
\end{thm}
\begin{proof}
We only proof for $\epsilon=+1$ case. The $\epsilon=-1$ can be proved
by applying the $*$-involution. We prove by induction on the length
$\ell\left(w\right)$. Let $w_{\geq2}=s_{i_{2}}\cdots s_{i_{\ell}}\in W$
and $\bm{i_{\geq2}}=\left(i_{2},\cdots,i_{\ell}\right)\in I\left(w_{\geq2}\right)$.
Let $b\in\mathscr{B}\left(\infty\right)$ with $L_{+1}\left(b,\bm{i}\right)=0$,
that is we have 
\[
\left(\varepsilon_{i_{1}}\left(b\right),\varepsilon_{i_{2}}\left(\sigma_{i_{1}}^{*}b\right),\cdots,\varepsilon_{i_{\ell}}\left(\sigma_{i_{\ell}}^{*}\cdots\sigma_{i_{1}}^{*}b\right)\right)=\left(0,\cdots,0\right),
\]
so let $b_{\geq2}:=\sigma_{i_{1}}^{*}b$, then we have 
\begin{align*}
L_{+1}\left(b_{\geq2},\bm{i}_{\geq2}\right) & =\left(\varepsilon_{i_{2}}\left(b_{\geq2}\right),\cdots,\varepsilon_{i_{\ell}}\left(\sigma_{i_{\ell}}^{*}\cdots\sigma_{i_{2}}^{*}b_{\geq2}\right)\right)\\
 & =\left(\varepsilon_{i_{2}}\left(\sigma_{i_{1}}^{*}b\right),\cdots,\varepsilon_{i_{\ell}}\left(\sigma_{i_{\ell}}^{*}\cdots\sigma_{i_{2}}^{*}\sigma_{i_{1}}^{*}b\right)\right)=\left(0,\cdots,0\right)\in\mathbb{Z}_{\geq0}^{\ell-1}
\end{align*}
by definition of the Lusztig datum.

By induction hypothesis, we have 
\[
f_{\epsilon}^{\up}\left(\bm{c}_{\geq2},\bm{i}_{\geq2}\right)G^{\up}\left(b_{\geq2}\right)-G^{\up}\left(\nabla_{\bm{i}_{\geq2},\epsilon}^{\bm{c}_{\geq2}}\left(b_{\geq2}\right)\right)\in\sum_{L_{\epsilon}\left(b_{\geq2}^{'},\bm{i}_{\geq2}\right)<\bm{c}_{\geq2}}q\mathbb{Z}\left[q\right]G^{\up}\left(b_{\geq2}^{'}\right).
\]
with $c_{\geq2}=\left(c_{2},\cdots,c_{\ell}\right)\in\mathbb{Z}_{\geq0}^{\ell-1}$.
Since $\Uq^{-}\cap T_{i_{1}}^{-1}\Uq^{-}$ is spanned by the dual
canonical basis $\left\{ G^{\up}\left(b\right)\mid\varepsilon_{i_{1}}^{*}\left(b\right)=0\right\} $
and $f_{\epsilon}^{\up}\left(\bm{c}_{\geq2},\bm{i}_{\geq2}\right)\in\Uq^{-}\cap T_{i_{1}}^{-1}\Uq^{-}$
and $G^{\up}\left(b_{\geq2}\right)\in\Uq^{-}\cap T_{i_{1}}^{-1}\Uq^{-}$,
so we obtain that $\varepsilon_{i_{1}}^{*}\left(\nabla_{\bm{i}_{\geq2},\epsilon}^{\bm{c}_{\geq2}}\left(b_{\geq2}\right)\right)=0$
and $\varepsilon_{i_{1}}^{*}\left(b_{\geq2}^{'}\right)=0$.

We have 
\begin{align*}
f_{+1}^{\up}\left(\bm{c},\bm{i}\right)G^{\up}\left(b\right) & =\left(1-q_{i_{1}}^{2}\right)^{\left\langle h_{i_{1}},\xi\left(c_{\geq2},i_{\geq2}\right)+\mathrm{wt}\left(b\right)\right\rangle }f_{i_{1}}^{\left\{ c_{1}\right\} }T_{i_{1}}\left(f_{\epsilon}^{\up}\left(\bm{c}_{\geq2},\bm{i}_{\geq2}\right)G^{\mathrm{up}}\left(b_{\geq2}\right)\right)\\
 & \in\left(1-q_{i_{1}}^{2}\right)^{\left\langle h_{i_{1}},\xi\left(c_{\geq2},i_{\geq2}\right)+\mathrm{wt}\left(b\right)\right\rangle }f_{i_{1}}^{\left\{ c_{1}\right\} }\\
 & \times T_{i_{1}}\left(G^{\up}\left(\nabla_{\bm{i}_{\geq2},+1}^{\bm{c}_{\geq2}}\left(b_{\geq2}\right)\right)+\sum_{L_{+1}\left(b_{\geq2}^{'},i_{\geq2}\right)<c_{\geq2}}q\mathbb{Z}\left[q\right]G^{\up}\left(b_{\geq2}^{'}\right)\right)\\
 & =f_{i_{1}}^{\left\{ c_{1}\right\} }\left(G^{\up}\left(\sigma_{i_{1}}\nabla_{\bm{i}_{\geq2},+1}^{\bm{c}_{\geq2}}\left(b_{\geq2}\right)\right)+\sum_{L_{+1}\left(b_{\geq2}^{'},i_{\geq2}\right)<c_{\geq2}}q\mathbb{Z}\left[q\right]G^{\up}\left(\sigma_{i_{1}}b_{\geq2}^{'}\right)\right).
\end{align*}
We note that $\tilde{f}_{i_{1}}^{c_{1}}\sigma_{i_{1}}\nabla_{\bm{i}_{\geq2},+1}^{\bm{c}_{\geq2}}\left(b_{\geq2}\right)=\tilde{f}_{i_{1}}^{c_{1}}\sigma_{i_{1}}\tilde{f}_{i_{2}}^{c_{1}}\sigma_{i_{2}}\cdots\tilde{f}_{i_{\ell-1}}^{c_{\ell-1}}\sigma_{i_{\ell-1}}\tilde{f}_{i_{\ell}}^{c_{\ell}}\sigma_{i_{\ell}}\sigma_{i_{\ell}}^{*}\cdots\sigma_{i_{2}}^{*}\left(b_{\geq2}\right)=\nabla_{\bm{i},+1}^{\bm{c}}\left(b\right)$
and 
\begin{align*}
f_{i_{1}}^{\left\{ c_{1}\right\} }G^{\up}\left(\sigma_{i_{1}}\nabla_{\bm{i}_{\geq2},+1}^{\bm{c}_{\geq2}}\left(b_{\geq2}\right)\right) & \in G^{\up}\left(\nabla_{\bm{i},+1}^{\bm{c}}\left(b\right)\right)+\sum_{\varepsilon_{i_{1}}\left(b''\right)<c_{1}}q\mathbb{Z}\left[q\right]G^{\up}\left(b''\right),\\
f_{i_{1}}^{\left\{ c_{1}\right\} }G^{\up}\left(\sigma_{i_{1}}b_{\geq2}^{'}\right) & \in G^{\up}\left(\tilde{f}_{i_{1}}^{c_{1}}\sigma_{i_{1}}b_{\geq2}^{'}\right)+\sum_{\varepsilon_{i_{1}}\left(b''\right)<c_{1}}q\mathbb{Z}\left[q\right]G^{\up}\left(b''\right)
\end{align*}
by Theorem \ref{thm:multdualChevalley}, $f_{i_{1}}^{\left\{ c_{1}\right\} }\left(G^{\up}\left(\sigma_{i_{1}}\nabla_{\bm{i}_{\geq2},+1}^{\bm{c}_{\geq2}}\left(b_{\geq2}\right)\right)+\sum_{L_{+1}\left(b_{\geq2}^{'},i_{\geq2}\right)<c_{\geq2}}q\mathbb{Z}\left[q\right]G^{\up}\left(\sigma_{i_{1}}b_{\geq2}^{'}\right)\right)$
can be written in the following form:
\[
G^{\up}\left(\nabla_{\bm{i},+1}^{\bm{c}}\left(b\right)\right)+\sum_{L_{+1}\left(b_{\geq2}^{'},i_{\geq2}\right)<\bm{c}_{\geq2}}q\mathbb{Z}\left[q\right]G^{\up}\left(\tilde{f}_{i_{1}}^{c_{1}}\sigma_{i_{1}}b_{\geq2}^{'}\right)+\sum_{\varepsilon_{i_{1}}\left(b''\right)<c_{1}}q\mathbb{Z}\left[q\right]G^{\up}\left(b''\right).
\]

Since we have $\left(c_{2}',\cdots,c_{\ell}'\right)=L_{+1}\left(b_{\geq2}^{'},i_{\geq2}\right)<\bm{c}_{\geq2}$,
we obtain $L_{+1}\left(\tilde{f}_{i_{1}}^{c_{1}}\sigma_{i_{1}}b_{\geq2}^{'},\bm{i}\right)=\left(c_{1},c_{2}',\cdots,c_{\ell}'\right)<\bm{c}$
and we have $L_{+1}\left(b'',\bm{i}\right)=\left(\varepsilon_{i_{1}}\left(b''\right),\cdots\right)<L_{+1}\left(b,\bm{i}\right)=\left(c_{1},c_{2},\cdots,c_{\ell}\right)$
by $\varepsilon_{i_{1}}\left(b^{''}\right)<c_{1}$. We we obtain the
claim. 
\end{proof}
Using the transition Theorem \ref{thm:qunip} (2) from \PBW basis
to the dual canonical basis, we obtain the following multiplicity-free
result.
\begin{thm}
Let $w\in W$, $\bm{i}=\left(i_{1},\cdots,i_{\ell}\right)\in I\left(w\right)$
and $\epsilon\in\left\{ \pm1\right\} $.\label{thm:main2}

For $\bm{c}\in\mathbb{Z}_{\geq0}^{\ell}$ and $b\in\mathscr{B}\left(>w,\epsilon\right)$,
we have 
\begin{align*}
G^{\up}\left(b_{\epsilon}\left(\bm{c},\bm{i}\right)\right)G^{\up}\left(b\right) & \in G^{\up}\left(\nabla_{\bm{i},\epsilon}^{\bm{c}}\left(b\right)\right)+\sum_{L_{\epsilon}\left(b^{'},\bm{i}\right)<c}q\mathbb{Z}\left[q\right]G^{\up}\left(b^{'}\right)\;\text{if}\;\epsilon=+1,\\
G^{\up}\left(b\right)G^{\up}\left(b_{\epsilon}\left(\bm{c},\bm{i}\right)\right) & \in G^{\up}\left(\nabla_{\bm{i},\epsilon}^{\bm{c}}\left(b\right)\right)+\sum_{L_{\epsilon}\left(b^{'},\bm{i}\right)<c}q\mathbb{Z}\left[q\right]G^{\up}\left(b^{'}\right)\;\text{if}\;\epsilon=-1.
\end{align*}

\end{thm}

\subsubsection{}
\begin{defn}
Let $w\in W$, $\bm{i}=\left(i_{1},\cdots,i_{\ell}\right)\in I\left(w\right)$
and $\epsilon\in\left\{ \pm1\right\} $. We define maps $\tau_{\leq w,\epsilon}\colon\mathscr{B}\left(\infty\right)\to\mathscr{B}\left(\leq w,\epsilon\right)$
and $\tau_{>w,\epsilon}\colon\mathscr{B}\left(\infty\right)\to\mathscr{B}\left(>w,\epsilon\right)$
by
\begin{align*}
\tau_{\leq w,\epsilon}\left(b\right) & =b_{\epsilon}\left(L_{\epsilon}\left(b,\bm{i}\right),\bm{i}\right),\\
\tau_{>w,\epsilon}\left(b\right) & =\begin{cases}
\sigma_{i_{1}}\cdots\sigma_{i_{\ell}}\hat{\sigma}_{i_{\ell}}^{*}\cdots\hat{\sigma}_{i_{1}}^{*}\left(b\right) & \text{if}\;\epsilon=+1,\\
\sigma_{i_{1}}^{*}\cdots\sigma_{i_{\ell}}^{*}\hat{\sigma}_{i_{\ell}}\cdots\hat{\sigma}_{i_{1}}\left(b\right) & \text{if}\;\epsilon=-1.
\end{cases}
\end{align*}
We note that the independence of the maps $\tau_{\leq w,\epsilon}$
and $\tau_{>w,\epsilon}$ on a reduced word $\bm{i}\in I\left(w\right)$
of a Weyl group element $w$ can be shown in \cite[Theorem 4.4]{BKT:affine}
using representation theory of the preprojective algebra and the torsion
pair in symmetric case. Using the folding argument , it also can be
proved in general.\end{defn}
\begin{prop}
We have a bijection as sets:
\[
\Omega_{w}:=\left(\tau_{\leq w,\epsilon},\tau_{>w,\epsilon}\right)\colon\mathscr{B}\left(\infty\right)\to\mathscr{B}\left(\leq w,\epsilon\right)\times\mathscr{B}\left(>w,\epsilon\right).
\]

\end{prop}
We prove the multiplication property of the dual canonical basis element
between $\Uq^{-}\left(\leq w,\epsilon\right)$ and $\Uq^{-}\left(>w,\epsilon\right)$.
\begin{thm}
Let $w\in W$, $\bm{i}=\left(i_{1},\cdots,i_{\ell}\right)\in I\left(w\right)$
and $\epsilon\in\left\{ \pm1\right\} $. For $b\in\mathscr{B}\left(\infty\right)$,
we have
\begin{align*}
G^{\up}\left(\tau_{\leq w,\epsilon}\left(b\right)\right)G^{\up}\left(\tau_{>w,\epsilon}\left(b\right)\right) & \in G^{\up}\left(b\right)+\sum_{L_{\epsilon}\left(b',\bm{i}\right)<L_{\epsilon}\left(b,\bm{i}\right)}q\mathbb{Z}\left[q\right]G^{\up}\left(b^{'}\right)\;\text{if}\;\epsilon=+1,\\
G^{\up}\left(\tau_{>w,\epsilon}\left(b\right)\right)G^{\up}\left(\tau_{\leq w,\epsilon}\left(b\right)\right) & \in G^{\up}\left(b\right)+\sum_{L_{\epsilon}\left(b',\bm{i}\right)<L_{\epsilon}\left(b,\bm{i}\right)}q\mathbb{Z}\left[q\right]G^{\up}\left(b^{'}\right)\;\text{if}\;\epsilon=-1.
\end{align*}
\end{thm}
\begin{proof}
We prove $\epsilon=+1$ case. We prove by the induction on the length
$\ell\left(w\right)$.

First we have 
\begin{align*}
 & G^{\up}\left(b\right)-f_{i_{1}}^{\left\{ \varepsilon_{i_{1}}\left(b\right)\right\} }G^{\up}\left(\tilde{e}_{i_{1}}^{\varepsilon_{i_{1}}\left(b\right)}b\right)\\
= & G^{\up}\left(b\right)-\left(1-q_{i_{1}}^{2}\right)^{\left\langle h_{i_{1}},\mathrm{wt}\left(\hat{\sigma}_{i_{1}}^{*}b\right)\right\rangle }f_{i_{1}}^{\left\{ \varepsilon_{i_{1}}\left(b\right)\right\} }T_{i_{1}}G^{\up}\left(\hat{\sigma}_{i_{1}}^{*}b\right)\in\sum_{\varepsilon_{i_{1}}\left(b'\right)<\varepsilon_{i_{1}}\left(b\right)}q\mathbb{Z}\left[q\right]G^{\up}\left(b'\right)
\end{align*}

By Theorem \ref{thm:main2}, we only have to compute the product $\left(1-q_{i}^{2}\right)^{\left\langle h_{i},\mathrm{wt}\left(\hat{\sigma}_{i}^{*}b\right)\right\rangle }f_{i}^{\left\{ \varepsilon_{i}\left(b\right)\right\} }T_{i}G^{\up}\left(\hat{\sigma}_{i}^{*}b\right)\times G^{\up}\left(\tau_{>w,+1}\left(b\right)\right)$.

We note that 
\[
G^{\up}\left(\tau_{>w,+1}\left(b\right)\right)=\left(1-q_{i}^{2}\right)^{\left\langle h_{i_{1}},\mathrm{wt}\left(\tau_{>w_{\geq2}}\left(\hat{\sigma}_{i_{1}}^{*}b\right)\right)\right\rangle }T_{i_{1}}G^{\up}\left(\tau_{>w_{\geq2}}\left(\hat{\sigma}_{i_{1}}^{*}b\right)\right)
\]
where $w_{\geq2}=s_{i_{2}}\cdots s_{i_{\ell}}$.

By the induction hypothesis, we have
\[
G^{\up}\left(\hat{\sigma}_{i_{1}}^{*}b\right)-G^{\up}\left(\tau_{\leq w_{\geq2}}\left(\hat{\sigma}_{i_{1}}^{*}b\right)\right)G^{\up}\left(\tau_{>w_{\geq2}}\left(\hat{\sigma}_{i_{1}}^{*}b\right)\right)\in\sum_{L_{+1}\left(b'',\bm{i}_{\geq2}\right)<L_{+1}\left(\hat{\sigma}_{i_{1}}^{*}b,\bm{i}_{\geq2}\right)}q\mathbb{Z}\left[q\right]G^{\up}\left(b''\right)
\]
where $i_{\geq2}=\left(i_{2},\cdots,i_{\ell}\right)\in I\left(s_{i_{2}}\cdots s_{i_{\ell}}\right)$.

Applying $\left(1-q_{i_{1}}^{2}\right)^{\left\langle h_{i_{1}},\mathrm{wt}\left(\hat{\sigma}_{i_{1}}^{*}b\right)\right\rangle }T_{i_{1}}$,
we obtain
\[
G^{\up}\left(\sigma_{i_{1}}\hat{\sigma}_{i_{1}}^{*}b\right)-G^{\up}\left(\sigma_{i_{1}}\tau_{\leq w_{\geq2}}\left(\hat{\sigma}_{i_{1}}^{*}b\right)\right)G^{\up}\left(\sigma_{i_{1}}\tau_{>w_{\geq2}}\left(\hat{\sigma}_{i_{1}}^{*}b\right)\right)\in\sum_{L_{+1}\left(b'',\bm{i}_{\geq2}\right)<L_{+1}\left(\hat{\sigma}_{i_{1}}^{*}b,\bm{i}_{\geq2}\right)}q\mathbb{Z}\left[q\right]G^{\up}\left(\sigma_{i_{1}}b''\right).
\]

We note that $\tilde{e}_{i_{1}}^{\varepsilon_{i_{1}}\left(b\right)}b=\sigma_{i_{1}}\hat{\sigma}_{i_{1}}^{*}b$.
Multiplying $f_{i_{1}}^{\left\{ \varepsilon_{i_{1}}\left(b\right)\right\} }$
from left to the second term, we have 
\begin{align*}
 & f_{i_{1}}^{\left\{ \varepsilon_{i_{1}}\left(b\right)\right\} }G^{\up}\left(\sigma_{i_{1}}\tau_{\leq w_{\geq2}}\left(\hat{\sigma}_{i_{1}}^{*}b\right)\right)G^{\up}\left(\sigma_{i_{1}}\tau_{>w_{\geq2}}\left(\hat{\sigma}_{i_{1}}^{*}b\right)\right)\\
\in & G^{\up}\left(\tilde{f}_{i_{1}}^{\varepsilon_{i_{1}}\left(b\right)}\sigma_{i_{1}}\tau_{\leq w_{\geq2}}\left(\hat{\sigma}_{i_{1}}^{*}b\right)\right)G^{\up}\left(\sigma_{i_{1}}\tau_{>w_{\geq2}}\left(\hat{\sigma}_{i_{1}}^{*}b\right)\right)+\sum_{\varepsilon_{i_{1}}\left(b'\right)<\varepsilon_{i_{1}}\left(b\right)}q\mathbb{Z}\left[q\right]G^{\up}\left(b'\right),
\end{align*}
then we obtain
\begin{gather*}
f_{i_{1}}^{\left\{ \varepsilon_{i_{1}}\left(b\right)\right\} }\left(G^{\up}\left(\tilde{e}_{i_{1}}^{\varepsilon_{i_{1}}\left(b\right)}b\right)-G^{\up}\left(\sigma_{i_{1}}\tau_{\leq w_{\geq2}}\left(\hat{\sigma}_{i_{1}}^{*}b\right)\right)G^{\up}\left(\sigma_{i_{1}}\tau_{>w_{\geq2}}\left(\hat{\sigma}_{i_{1}}^{*}b\right)\right)\right)\\
\in\sum_{\varepsilon_{i_{1}}\left(b'\right)<\varepsilon_{i_{1}}\left(b\right)}q\mathbb{Z}\left[q\right]G^{\up}\left(b'\right)+\sum_{L_{+1}\left(b'',\bm{i}_{\geq2}\right)<L_{+1}\left(\hat{\sigma}_{i_{1}}^{*}b,\bm{i}_{\geq2}\right)}q\mathbb{Z}\left[q\right]G^{\up}\left(\tilde{f}_{i_{1}}^{\varepsilon_{i_{1}}\left(b\right)}\sigma_{i_{1}}b''\right).
\end{gather*}

By the construction, we have $\tilde{f}_{i_{1}}^{\varepsilon_{i_{1}}\left(b\right)}\sigma_{i_{1}}\tau_{\leq w_{\geq2}}\left(\hat{\sigma}_{i_{1}}^{*}b\right)=\tau_{\leq w}\left(b\right)$
and $\sigma_{i_{1}}\tau_{>w_{\geq2}}\left(\hat{\sigma}_{i_{1}}^{*}b\right)=\tau_{>w}\left(b\right)$,
hence we obtain the claim of the theorem.

\begin{NB}

The following is the original rough computation. Essentially same
as above.

Let $L_{\epsilon}\left(b,\bm{i}\right)=\left(c_{1},\cdots,c_{\ell}\right)$.
Since we have 
\begin{align*}
G^{\up}\left(\tau_{\leq w,+1}\left(b\right)\right) & =\\
 & f_{i_{1}}^{\left[c_{1}\right]}\left(1-q_{i_{1}}^{2}\right)^{\left\langle h_{i_{1}},\xi\left(\bm{c}_{\geq2},\bm{i}_{\ge2}\right)\right\rangle }T_{i_{1}}G^{\up}\left(\tau_{\leq w_{\geq2}}\left(b\right)\right)+\sum_{b^{'}}q\mathbb{Z}\left[q\right]G^{\mathrm{up}}\left(b'\right)
\end{align*}

We have 
\begin{align*}
G^{\up}\left(\tau_{\leq w}\left(b\right)\right)\cdot G^{\up}\left(\tau_{>w}\left(b\right)\right) & =\left(f_{i_{1}}^{\left[c_{1}\right]}\left(1-q_{i_{1}}^{2}\right)^{?}T_{i_{1}}G^{\up}\left(\tau_{\leq\bm{w_{\geq2}}}\left(\hat{\sigma}_{i_{1}}^{*}b\right)\right)\right)\cdot G^{\up}\left(\tau_{>w}\left(b\right)\right)+?\\
 & =f_{i_{1}}^{\left[c_{1}\right]}\left(1-q_{i_{1}}^{2}\right)^{?+?}T_{i_{1}}\left(G^{\up}\left(\tau_{\leq\bm{w_{\geq2}}}\left(\hat{\sigma}_{i_{1}}^{*}b\right)\right)\cdot G^{\up}\left(\sigma_{i_{1}}^{*}\tau_{>w}\left(b\right)\right)\right)+?\\
 & =f_{i_{1}}^{\left[c_{1}\right]}\left(1-q_{i_{1}}^{2}\right)^{?+?}T_{i_{1}}\left(G^{\up}\left(\tau_{\leq w_{\geq2}}\left(\hat{\sigma}_{i_{1}}^{*}b\right)\right)G^{\up}\left(\tau_{>\bm{w_{\geq2}}}\left(\hat{\sigma}_{i_{1}}^{*}\left(b\right)\right)\right)\right)+?\\
 & =f_{i_{1}}^{\left[c_{1}\right]}\left(1-q_{i_{1}}^{2}\right)^{?+?}T_{i_{1}}\left(G^{\up}\left(\hat{\sigma}_{i_{1}}^{*}\left(b\right)\right)+\sum_{L_{\bm{w_{\geq2}}}\left(b^{'}\right)<_{w_{\geq2}}L_{\bm{w_{\geq2}}}\left(\hat{\sigma}_{i_{1}}^{*}\left(b\right)\right)}q\mathbb{Z}\left[q\right]G^{\up}\left(b^{'}\right)\right)+?\\
 & =f_{i_{1}}^{\left[c_{1}\right]}\left(G^{\up}\left(\sigma_{i_{1}}\hat{\sigma}_{i_{1}}^{*}\left(b\right)\right)+\sum_{L_{\bm{w_{\geq2}}}\left(b^{'}\right)<_{w_{\geq2}}L_{\bm{w_{\geq2}}}\left(\hat{\sigma}_{i_{1}}^{*}\left(b\right)\right)}q\mathbb{Z}\left[q\right]G^{\up}\left(\sigma_{i_{1}}b^{'}\right)\right)+?\\
 & =G^{\up}\left(\tilde{f}_{i_{1}}^{c_{1}}\sigma_{i_{1}}\hat{\sigma}_{i_{1}}^{*}\left(b\right)\right)+\sum_{L_{\bm{w_{\geq2}}}\left(b^{'}\right)<_{w_{\geq2}}L_{\bm{w_{\geq2}}}\left(\hat{\sigma}_{i_{1}}^{*}\left(b\right)\right)}q\mathbb{Z}\left[q\right]G^{\up}\left(\tilde{f}_{i_{1}}^{c_{1}}\sigma_{i_{1}}b^{'}\right)+??\\
\end{align*}

\end{NB}
\end{proof}

\subsection{Application}

We give a slight refinement of Lusztig's result \cite[Proposition 8.3]{Lus:braid}
in the dual canonical basis. The following can be shown in a similar
manner using the multiplicity-free property of the multiplications
of a triple of the dual canonical basis elements, so we only state
the claims.
\begin{thm}
Let $w$ be a Weyl group element, $\bm{i}=\left(i_{1},\cdots,i_{\ell}\right)\in I\left(w\right)$
and $p\in\left[0,\ell\right]$ be an integer. We consider the following
intersection:
\[
\Uq^{-}\cap T_{s_{i_{p+1}}\cdots s_{i_{\ell}}}\Uq^{-}\cap T_{s_{i_{1}}\cdots s_{i_{p}}}^{-1}\Uq^{-}=\left(\Uq^{-}\cap T_{s_{i_{p+1}}\cdots s_{i_{\ell}}}\Uq^{-}\right)\cap\left(\Uq^{-}\cap T_{s_{i_{1}}\cdots s_{i_{p}}}^{-1}\Uq^{-}\right)
\]

\textup{(1)} The subalgebra $\Uq^{-}\cap T_{s_{i_{p+1}}\cdots s_{i_{\ell}}}\Uq^{-}\cap T_{s_{i_{1}}\cdots s_{i_{p}}}^{-1}\Uq^{-}$
is compatible with the dual canonical basis, that is there exists
a subset $\mathscr{B}\left(\Uq^{-}\cap T_{i_{p+1}}\cdots T_{i_{\ell}}\Uq^{-}\cap T_{s_{i_{1}}\cdots s_{i_{p}}}^{-1}\Uq^{-}\right)\subset\mathscr{B}\left(\infty\right)$
such that 
\[
\Uq^{-}\cap T_{s_{i_{p+1}}\cdots s_{i_{\ell}}}\Uq^{-}\cap T_{s_{i_{1}}\cdots s_{i_{p}}}^{-1}\Uq^{-}=\bigoplus_{b\in\mathscr{B}\left(\Uq^{-}\cap T_{s_{i_{p+1}}\cdots s_{i_{\ell}}}\Uq^{-}\cap T_{s_{i_{1}}\cdots s_{i_{p}}}^{-1}\Uq^{-}\right)}\mathbb{Q}\left(q\right)G^{\up}\left(b\right).
\]

\textup{(2)} Multiplication in $\Uq^{-}\left(\mathfrak{g}\right)_{\mathcal{A}}^{\up}$
defines an isomorphism of free $\mathcal{A}$-modules:
\[
\left(\Uq^{-}\left(s_{i_{p+1}}\cdots s_{i_{\ell}},+1\right)\right)_{\mathcal{A}}^{\up}\otimes_{\mathcal{A}}\left(\Uq^{-}\cap T_{s_{i_{p+1}}\cdots s_{i_{\ell}}}\Uq^{-}\cap T_{s_{i_{1}}\cdots s_{i_{\ell}}}^{-1}\Uq^{-}\right)_{\mathcal{A}}^{\up}\otimes_{\mathcal{A}}\left(\Uq^{-}\left(s_{i_{p}}\cdots s_{i_{1}},-1\right)\right)_{\mathcal{A}}^{\up}\to\Uq^{-}
\]
where $\left(\Uq^{-}\cap T_{i_{p+1}}\cdots T_{i_{\ell}}\Uq^{-}\cap T_{i_{p}}^{-1}\cdots T_{i_{1}}^{-1}\Uq^{-}\right)_{\mathcal{A}}^{\up}=\Uq^{-}\left(\mathfrak{g}\right)_{\mathcal{A}}^{\up}\cap T_{i_{p+1}}\cdots T_{i_{\ell}}\Uq^{-}\cap T_{i_{p}}^{-1}\cdots T_{i_{1}}^{-1}\Uq^{-}$.
\end{thm}
\begin{NB}

\subsubsection{}

Using the multiplication formula, we obtain the following integrality
of the bilinear form which is conjectured by Berenstein and Greenstein.
\begin{lem}
For $w\in W$, $\bm{c}\in\mathbb{Z}_{\geq0}^{\ell}\setminus\left\{ \bm{0}\right\} $
and $\epsilon\in\left\{ \pm1\right\} $, we have 
\[
\left(f_{\epsilon}\left(\bm{c},\bm{i}\right),G^{\mathrm{up}}\left(b\right)\right)=0
\]
 for $b\in\mathscr{B}\left(\mathbf{U}_{q}^{-}\cap T_{w^{\epsilon}}^{\epsilon}\mathbf{U}_{q}^{-}\right)$.\end{lem}
\begin{proof}
Applying the $*$-involution, we obtain the claim for $\epsilon=-1$
case from $\epsilon=+1$ case. So we only have to prove the claim
for $\epsilon=+1$ case. We prove the claim by the induction on $\ell\left(w\right)$.
If $\ell\left(w\right)=1$, this follows from the orthogonal decomposition
\[
\mathbf{U}_{q}^{-}=f_{i}\mathbf{U}_{q}^{-}\oplus\left(\mathbf{U}_{q}^{-}\cap T_{i}\mathbf{U}_{q}^{-}\right)
\]
and the compatibility. For $\ell\left(w\right)>1$, we assume that
$c_{1}>0$, then it reduces to the first case, so it is done. We assume
that $c_{1}=0$ and set $\bm{c}_{\geq2}=\left(c_{2},\cdots,c_{\ell-1}\right)$
and $i_{\geq2}=\left(i_{2},\cdots,i_{\ell}\right)\in I\left(s_{i_{2}}\cdots s_{i_{\ell}}\right)$
and $b_{\geq2}:=\sigma_{i_{1}}^{*}b\in\mathscr{B}\left(\mathbf{U}_{q}^{-}\cap T_{i_{2}}\mathbf{U}_{q}^{-}\right)$.
We have 
\begin{align*}
\left(f_{+1}\left(\bm{c},\bm{i}\right),G^{\mathrm{up}}\left(b\right)\right) & =\left(1-q_{i_{1}}^{2}\right)^{\left\langle h_{i_{1}},\mathrm{wt}\left(b_{\geq2}\right)\right\rangle }\left(T_{i_{1}}\left(f_{+1}\left(\bm{c}_{\geq2},\bm{i}_{\geq2}\right)\right),T_{i_{1}}G^{\mathrm{up}}\left(b_{\geq2}\right)\right)\\
 & =\left(f_{+1}\left(\bm{c}_{\geq2},\bm{i}_{\geq2}\right),G^{\mathrm{up}}\left(b_{\geq2}\right)\right).
\end{align*}
Since $\bm{c}_{\geq2}\in\mathbb{Z}_{\geq0}^{\ell-1}\setminus\left\{ \bm{0}\right\} $,
we get the claim by the induction.\end{proof}
\begin{prop}
We have $\left(\mathbf{U}_{q}^{-}\left(w\right)_{\mathbb{Z}},\mathbf{U}_{q}^{-}\left(\mathfrak{g}\right)\right)\subset\mathbb{Z}\left[q^{\pm1}\right]$.\end{prop}
\begin{proof}
For $b\in\mathscr{B}\left(\infty\right)$ and $\bm{c}=\left(c_{1},\cdots,c_{\ell}\right)\in\mathbb{Z}_{\geq0}^{\ell}$,
we prove $\left(f_{+}^{\mathrm{up}}\left(\bm{c},\bm{i}\right),G^{\mathrm{up}}\left(b\right)\right)\in\mathbb{Z}\left[q^{\pm1}\right]$.

We prove by the induction on the height of $b$ and the $\bm{i}$-Lusztig
datum for $b\in\mathscr{B}\left(\infty\right)$ with respect to the
lexicographic order.

If $L_{+}\left(b,\bm{i}\right)=\bm{0}\in\mathbb{Z}_{\geq0}^{\ell}$,
we have the claim by the above lemma.

We assume that $L_{+}\left(b,\bm{i}\right)>\bm{0}$. Let $\left(\tau_{\leq w}b,\tau_{>w}b\right)\in\mathscr{B}\left(\mathbf{U}_{q}^{-}\left(w\right)\right)\times\mathscr{B}\left(\mathbf{U}_{q}^{-}\cap T_{w}\mathbf{U}_{q}^{-}\right)$
with $L_{+}\left(\tau_{\leq w}b,\bm{i}\right)=L_{+}\left(b,\bm{i}\right)>0$
by the assumption on $b$. Since we have the multiplication formula,
it suffices for us to prove that $\left(f_{+}^{\mathrm{up}}\left(\bm{c},\bm{i}\right),G^{\mathrm{up}}\left(\tau_{\leq w}b\right)G^{\mathrm{up}}\left(\tau_{>w}b\right)\right)$
for $\bm{c}\in\mathbb{Z}_{\geq0}^{\ell}$ by the induction hypothesis.
We have 
\[
\left(f_{+}^{\mathrm{up}}\left(\bm{c},\bm{i}\right),G^{\mathrm{up}}\left(\tau_{\leq w}b\right)G^{\mathrm{up}}\left(\tau_{>w}b\right)\right)=\left(r\left(f_{+}^{\mathrm{up}}\left(\bm{c},\bm{i}\right)\right),G^{\mathrm{up}}\left(\tau_{\leq w}b\right)\otimes G^{\mathrm{up}}\left(\tau_{>w}b\right)\right).
\]
Using the right coideal property $r\left(f_{+}^{\mathrm{up}}\left(\bm{c},\bm{i}\right)\right)\in\mathbf{U}_{q}^{-}\left(\mathfrak{g}\right)_{\mathbb{Z}}^{\mathrm{up}}\otimes\mathbf{U}_{q}^{-}\left(w\right)_{\mathbb{Z}}^{\mathrm{up}}$,
we obtain the claim by the induction hypothesis since $\mathrm{ht}\left(\tau_{\leq w}b\right),\mathrm{ht}\left(\tau_{>w}b\right)<\mathrm{ht}\left(b\right)$
by the assumption $L_{+}\left(\tau_{\leq w}b,\bm{i}\right)>0$. 
\end{proof}
\end{NB}

\begin{NB}

\subsection{Conjectures on the intersection $\Uq^{-}\cap T_{w_{1}^{-1}}^{-1}\Uq^{-}\cap T_{w_{2}^{-1}}^{-1}\Uq^{\geq0}$}

For a pair $\left(w_{1},w_{2}\right)$ of Weyl group elements, we
consider the intersection $\Uq^{-}\cap T_{w_{1}^{-1}}^{-1}\Uq^{-}\cap T_{w_{2}^{-1}}^{-1}\Uq^{\geq0}$.

If $w_{1}$ is a left factor of $w_{2}$, that is $\ell\left(w_{1}\right)+\ell\left(w_{1}^{-1}w_{2}\right)=\ell\left(w_{2}\right)$,
it can be shown that the intersection $\Uq^{-}\cap T_{w_{1}^{-1}}^{-1}\Uq^{-}\cap T_{w_{2}^{-1}}^{-1}\Uq^{\geq0}$
has a structure of quantum cluster algebra using \GLS \cite{GLS:qcluster}
and also Goodearl-Yakimov \cite{GY:2013_QCAQNA} since it has a structure
of a quantum nilpotent algebra by the Levendorskii-Soibelman formula.

In view of Leclerc \cite{Lec:strata}, without the condition $\ell\left(w_{1}\right)+\ell\left(w_{1}^{-1}w_{2}\right)=\ell\left(w_{2}\right)$,
the following first claim seems to be proved when $\mathfrak{g}$
is a simply-laced type. Also in view of the results on the monoidal
categorification by Kang-Kashiwara-Kim-Oh \cite{KKKO:2014monoidal,KKKO:2015monoidal}
, the following result seems to be proved if $\mathfrak{g}$ is symmetric.
\begin{conjecture}
\textup{(1)} If $w_{1}\leq w_{2}$ in the Bruhat order, the triple
intersection $\Uq^{-}\cap T_{w_{1}^{-1}}^{-1}\Uq^{-}\cap T_{w_{2}^{-1}}^{-1}\Uq^{\geq0}$
contains a quantum cluster algebra.

\textup{(2)} the set of quantum cluster monomials is contained the
dual canonical basis $\mathbf{B}^{\up}\cap\Uq^{-}\cap T_{w_{1}}\Uq^{-}\cap T_{w_{2}}\Uq^{\geq0}$.
\end{conjecture}
The relation between Lenagan-Yakimov's approach \cite{LY:2015qRic}
to quantum Richardson varieties should be also examined in future.

\end{NB}

\bibliographystyle{amsplain}
\bibliography{qunip}

\end{document}